\documentclass[a4paper, 10pt, DIV10, headinclude=false, footinclude=false]{scrartcl}

\usepackage{amsthm, amsmath, amssymb, amsfonts}
\usepackage[utf8]{inputenc}
\usepackage[english]{babel}
\usepackage{url}
\usepackage{mathtools}

\usepackage{subcaption}
\usepackage{xcolor}

\newtheorem{theorem}{Theorem}
\newtheorem{proposition}[theorem]{Proposition}

\theoremstyle{definition}

\newtheorem{lemma}[theorem]{Lemma}

\theoremstyle{remark}
\newtheorem{remark}[theorem]{Remark}
\newtheorem{assumption}[theorem]{Assumption}

\numberwithin{theorem}{section}
\numberwithin{figure}{section}
\numberwithin{table}{section}
\numberwithin{equation}{section}

\newenvironment{abstr}[1]{ \vspace{.05in}\footnotesize
	\parindent .2in
	{\upshape\bfseries #1. }\ignorespaces}{\par\vspace{.1in}}
\newenvironment{Abstract}{\begin{abstr}{Abstract}}{\end{abstr}}
\newenvironment{keywords}{\begin{abstr}{Key words}}{\end{abstr}}
\newenvironment{AMS}{\begin{abstr}{AMS subject classifications}}{\end{abstr}}

\newcommand{\Cdata}{C_{\textnormal{data}}}
\newcommand{\uit}[1]{u^{(#1)}}
\newcommand{\uhit}[1]{u_{h,p}^{(#1)}}

\allowdisplaybreaks[4]

\begin{document}
	\title{Higher-order finite element methods for the nonlinear Helmholtz equation}
	\author{Barbara Verf\"urth\thanks{Institut f\"ur Numerische Simulation, Universit\"at Bonn, Friedrich-Hirzebruch-Allee 7, D-53115 Bonn, Germany}}
	\date{}
	
	\maketitle

	\begin{Abstract}
		In this work, we analyze the finite element method with arbitrary but fixed polynomial degree for the nonlinear Helmholtz equation with impedance boundary conditions.
		We show well-posedness and error estimates of the finite element solution under a resolution condition between the wave number $k$, the mesh size $h$ and the polynomial degree $p$ of the form ``$k(kh)^{p+1}$ sufficiently small'' and a so-called smallness of the data assumption. For the latter, we prove that the logarithmic dependence in $h$ from the case $p=1$ in [H.~Wu, J.~Zou, \emph{SIAM J.~Numer.~Anal.} 56(3): 1338-1359, 2018] can be removed for $p\geq 2$.
		We show convergence of two different fixed-point iteration schemes. Numerical experiments illustrate our theoretical results and compare the robustness of the iteration schemes with respect to the size of the nonlinearity and the right-hand side data.
	\end{Abstract}
	
	\section{Introduction}
In various situations, such as for high intensities, linear(ized) material laws are no longer accurate enough and nonlinear constitutive relations have to be incorporated into the models.
One well-known example are Kerr-type materials \cite{Ker75} in electromagnetics, where the permittivity $\varepsilon$ depends on the electric field $E$ like $\varepsilon(E)=\varepsilon_0+\varepsilon_2|E|^2$.
In general, wave propagation in nonlinear media causes a lot of new possible phenomena such as optical bistability \cite{GolG84}.

As simplified model, we study in the following the nonlinear Helmholtz problem
\begin{align*}
-\Delta u -k^2(1+\varepsilon \chi_D|u|^2) u &= f \qquad \text{in}\quad \Omega,\\
\partial_\nu u+iku&=g\qquad \text{on}\quad \partial \Omega,
\end{align*}
where $D\subset\subset \Omega$ is the subdomain where the nonlinearity is ``active''.
Detailed assumptions on the domain and the data are given below.
Problems of the above form occur in nonlinear acoustics as well as in time-harmonic and suitably polarized nonlinear electromagnetics.

The nonlinear Helmholtz equation has been studied analytically for instance in \cite{EveW14}. Various numerical approaches have been suggested as well: \cite{BarFT09} and \cite{XuB10} consider layered media and study a finite volume approach or approximate it as the steady state to a Schrödinger equation, respectively. \cite{YuaL17} focuses on different iteration schemes for the nonlinearity and uses a pseudospectral method in space. 
A multiscale finite element method is proposed and analyzed for the heterogeneous nonlinear Helmholtz equation in \cite{MaiV22}.
The present work is inspired by \cite{WuZ18}, where the \emph{linear (i.e., $p=1$)} finite element method (FEM) is studied and a priori error estimates are shown. The main findings of \cite{WuZ18} are that, under a smallness of the data assumption between $k$, $\varepsilon$, $f$ and $g$ as well as the resolution condition $k^3h^2$ sufficiently small, a unique finite element solution exists and the discretization error is of the order $kh+k^3h^2$.
We emphasize that both the resolution condition as well as the a priori error estimate are similar to the well-studied linear case in the so-called pre-asymptotic regime, cf.~\cite{DuWu2015}.
In a similar spirit as \cite{WuZ18}, the recent work \cite{JiLiWuZo2022} provides a finite element error analysis of the nonlinear Helmholtz equation with perfectly matched layer at the boundary and Newton's method as iteration scheme.

In fact, a major ingredient in the numerical analysis of \cite{WuZ18,JiLiWuZo2022} is the study of an auxiliary linearized Helmholtz problem which is solved in each iteration step.
The finite element error analysis of the linear Helmholtz equation is much more mature than of its nonlinear counterpart.
Seminal results are the asymptotic $hp$-FEM analysis of \cite{MelS10,MelS11} and the pre-asymptotic error analysis for arbitrary, but fixed polynomial degree of \cite{DuWu2015}. These results have been obtained for the constant coefficient case, but recently much progress has been made for the heterogeneous Helmholtz equation as well.  In the asymptotic regime, arbitrary but fixed polynomial degree is treated in \cite{ChNi2020} and the $hp$-FEM in \cite{GraS20,LaSpWu2021a,BeChMe2021}. Pre-asymptotic estimates for the absolute error can be found in \cite{GaSp23,Pem2020}, whereas \cite{LaSpWu2022} studies the relative error.
By the difference between ``arbitrary but fixed polynomial degree'' and ``$hp$-FEM'' results, we mean that in the first case, constants may (implicitly) depend on the polynomial degree.
By now higher-order- and $hp$-FEM approximations are the state of the art -- in comparison to linear FEM -- for the Helmholtz equation as they allow a relaxed resolution condition of $k(kh)^{2p}$ sufficiently small (pre-asymptotic, fixed polynomial degree) or $kh/p$ sufficiently small and $p\gtrsim \ln k$ (asymptotic, $hp$ version).

Our main contribution is the rigorous a priori error analysis of higher-order finite element methods for the \emph{nonlinear} Helmholtz problem.
We essentially show that under a smallness of the data assumption similar to \cite{WuZ18,JiLiWuZo2022}, a resolution condition of $k(kh)^{p+1}$ is sufficient for existence and uniqueness of a finite element solution and the discretization error is of the order $h+(kh)^{p}+k(kh)^{2p}$.
We also rely on the numerical analysis of a linearized Helmholtz equation, where we prove a solution splitting (into an analytic and a less oscillatory part) in the spirit of \cite{MelS10,MelS11} as well as pre-asymptotic stability and error estimates. 
The linearized Helmholtz equation has a non-constant, discontinuous refractive index, so that we cannot directly apply recent results for the heterogeneous Helmholtz equation \cite{LaSpWu2021b,BeChMe2021,Pem2020}. If we assume smoothness of the domain where the nonlinearity is active, one might transfer the recent results from \cite{GaSp23}. However, here we use a  perturbation argument that the deviation from the constant-coefficient case is sufficiently small due to the smallness of the data assumption. Moreover, we show a discrete stability result in the $L^\infty$-norm which causes the tighter resolution condition $k(kh)^{p+1}\lesssim 1$ compared to the linear case.
Along this analysis, we treat two different iteration schemes: the frozen nonlinearity scheme considered in \cite{WuZ18} and a scheme suggested in \cite{YuaL17}. For the latter, we provide the first proof of linear convergence by interpreting it as a fixed-point iteration. This result fills a theoretical gap in \cite{YuaL17} and may be of own interest.

The paper is organized as follows. We present the setting and study the iteration schemes in Section \ref{sec:setting}. Section \ref{sec:nonlineardiscrete} describes the finite element discretization and the main error estimates, whose proofs are then presented in Section \ref{sec:proofs}. Finally, we illustrate our theoretical results with numerical experiments in Section \ref{sec:numexp}, where we also compare the two iteration schemes numerically in detail.

\section{Nonlinear Helmholtz equation in the continuous setting}\label{sec:setting}
In this section, we formulate our model problem and discuss the solution of the nonlinear problem via iteration schemes in the continuous setting.

Throughout this article, all our functions are complex-valued unless otherwise mentioned.
For any (sub)domain $S$, $(\cdot, \cdot)_S$ denotes the complex $L^2$-scalar product (with complex conjugate in the second argument).
We use standard notation on Sobolev spaces $H^s(S)$ and their norms.
Further, we use the following (semi)norms $\|\cdot\|_{0,S} := \|\cdot\|_{L^2(S)}$, $|\cdot|_{1,S} := \|\nabla \cdot\|_{0,S}$, and $\|\cdot \|_{2,S}:=\|\cdot \|_{H^2(S)}$.
As usual in the Helmholtz context, we also employ the following $k$-weighted norm $\|\cdot\|^2_{1,k,S} := |\cdot|_{1,S}^2 + k^2\,\|\cdot\|_{0,S}^2$.
We will omit the subdomain $S$ in the notation of norms and scalar products if it equals the full computational domain $\Omega$ and no confusion can arise.
Last, we use the notation $a \lesssim b$ to indicate that there exists a generic constant $C$, independent of $h$ and $k$ but possibly dependent on the polynomial degree $p$, such that $a \leq C b$.
\subsection{Model problem}\label{subsec:problem}
Let $\Omega \subset \mathbb{R}^d$, $d \in \{2,3\}$, be a bounded star-shaped domain with analytic boundary $\Gamma = \partial \Omega$ and outer normal $\nu$. 
In this work, we are interested in approximating the (weak) solution $u\in H^1(\Omega)$ of the following nonlinear Helmholtz problem
\begin{equation}\label{eq:NLHweak}
\mathcal B(u,v) := (\nabla u, \nabla v) -  (k^2(1+\chi_D\varepsilon|u|^2) u,v) + i (k u,v)_\Gamma = (f,v) + (g, v)_\Gamma
\end{equation}
for all $v \in H^1(D)$.
Here, $k$ is the wave number, $\varepsilon$ the Kerr coefficient, and $\chi_D$ denotes the characteristic function of $D$.  
We make the following assumptions on the data throughout the whole article.
\begin{assumption}
	\label{a:coeff}
	We assume that $f\in L^2(\Omega)$, $g\in H^{1/2}(\Gamma)$ and $\varepsilon \in \mathbb{R}_{>0}$. Suppose that $k\gtrsim 1$ in the sense there exists a constant $k_0>0$ such that $k\geq k_0$ and, subsequently, all constants in our estimates may depend on $k_0$. Finally, we assume that $D\subset\subset \Omega$ is a non-empty compactly embedded subdomain with Lipschitz boundary.
\end{assumption}
In the following, we abbreviate $\Cdata\coloneqq \|f\|_0+\|g\|_{H^{1/2}(\Gamma)}$.
\cite{WuZ18} shows that there exists $\theta_0$ such that if \begin{equation}\label{eq:smallnessdatacont}
k^{d-2}\varepsilon\Cdata^2<\theta_0,
\end{equation}
there exists a unique solution $u\in H^1(\Omega)$ to \eqref{eq:NLHweak}. Further, $u$ satisfies the following a priori estimates
\begin{equation}\label{eq:apriorinonlin}
\|u\|_{1,k}\lesssim \Cdata,\qquad \|u\|_2\lesssim k\Cdata,\qquad \|u\|_{L^\infty(D)}\lesssim k^{(d-3)/2}\Cdata.
\end{equation}
The (sufficient) condition $k^{d-2}\varepsilon\Cdata^2<\theta_0$ for these results to hold is called a smallness of the data assumption and it comes from a Banach fixed-point argument, see also Section \ref{subsec:iter} below.
While the exact condition itself does not seem to be sharp in numerical experiments, it is well known that some condition on $k, \varepsilon,f$ and $g$ is required for uniqueness. 
Note that the power of $k$ in the smallness assumption depends on the stability constant for the linear Helmholtz equation with $\varepsilon =0$. Throughout this article, we assume this stability constant to be $O(1)$, which is well established for star-shaped domains $\Omega$, see, e.g., \cite{Mel95,CuFe2006}.

\begin{remark}
	For some results, the assumption $g\in L^2(\Gamma)$ instead of $g\in H^{1/2}(\Gamma)$ would be sufficient. For simplicity, we omit to track this and work under Assumption \ref{a:coeff} and with the constant $\Cdata$.
\end{remark}

\subsection{Iteration schemes}\label{subsec:iter}
We present and discuss two iteration schemes for the nonlinear problem in the continuous setting. Both schemes will subsequently be combined with the spatial discretization in Section \ref{subsec:fem} to obtain a discrete solution in practice.

\cite{WuZ18} considers the following fixed-point iteration based on a \emph{frozen nonlinearity} approach. Given the previous iterate $\uit{l-1}$, the next iterate $\uit{l}\in H^1(\Omega)$ is defined as
\begin{equation}\label{eq:froznonlinitcont}
\begin{aligned}
\mathcal{B}_{\mathrm{lin}}(\uit{l-1}; \uit{l}, v)&=(f,v)+(g, v)_\Gamma \qquad \text{for all}\quad v\in H^1(\Omega),\\
\text{where}\quad \mathcal{B}_{\mathrm{lin}}(\Phi; v,w)&\coloneqq (\nabla v, \nabla w) -  (k^2(1+\chi_D\varepsilon|\Phi|^2) v,w) + i (k v,w)_\Gamma.
\end{aligned}
\end{equation}
The iteration starts from some $\uit{0}\in H^1(\Omega)\cap L^\infty(D)$ with sufficiently small energy norm. For simplicity, we consider $\uit{0}\equiv 0$ throughout.
Under the smallness of data assumption \eqref{eq:smallnessdatacont} with suitably chosen $\theta_0$, the sequence $\{\uit{l}\}_{l\in \mathbb N}$ forms a strict contraction in the sense that
\begin{equation}\label{eq:contractionfrozennonlin}
\|\uit{l+1}-\uit{l}\|_{1,k}\leq \frac 12 \|\uit{l}-\uit{l-1}\|_{1,k}\qquad \forall l\geq 1.
\end{equation}
This is the main idea in the proof of existence and uniqueness of a solution $u$ to \eqref{eq:NLHweak} in \cite{WuZ18}. Precisely, the sequence $\{\uit{l}\}$ converges to $u$ strongly in $H^1(\Omega)$. 

\cite{YuaL17} proposes a different iteration scheme -- the motivation stems from Newton's method, but it is neither Newton's method itself nor any variant thereof.
Given the previous iterate $\uit{l-1}$, the next iterate $\uit{l}\in H^1(\Omega)$ is defined as
\begin{equation}\label{eq:simplenewtonitcont}
\begin{aligned}
\mathcal{A}(\uit{l-1}; \uit{l}, v)&=-k^2\varepsilon(|\uit{l-1}|^2\uit{l-1}, v)_D+(f,v)+(g, v)_\Gamma \quad \text{for all }v\in H^1(\Omega),\\
\text{where}\quad \mathcal{A}(\Phi; v,w)&\coloneqq (\nabla v, \nabla w) -  (k^2(1+2\chi_D\varepsilon|\Phi|^2) v,w) + i (k v,w)_\Gamma.
\end{aligned}
\end{equation}
Again, we let the iteration start from $\uit{0}\equiv 0$ for simplicity.
\cite{YuaL17} observes numerically that the scheme converges with a linear rate and that it has the advantage of allowing larger values of $k$, $\varepsilon$ and the data than \eqref{eq:froznonlinitcont}.
To the best of our knowledge, these observations have not been rigorously confirmed in theory.
In the rest of this section, we will hence address the convergence of \eqref{eq:simplenewtonitcont} and relate it to the convergence of the frozen nonlinearity approach.

\subsection{Linearized Helmholtz equation as auxiliary problem}\label{subsec:linearized}
Both iteration schemes \eqref{eq:froznonlinitcont} and \eqref{eq:simplenewtonitcont} solve a linear(ized) Helmholtz
problem in each step, which we can formulate as follows.
Let $\Phi\in H^1(\Omega)\cap L^\infty(D)$ be given. 
For \eqref{eq:froznonlinitcont}, find $w\in H^1(\Omega)$ such that
\begin{equation}\label{eq:auxpb}
\mathcal{B}_\mathrm{lin}(\Phi; w, v):=(\nabla w, \nabla v)-k^2((1+\chi_D\varepsilon|\Phi|^2 )w, v)+ik(w,v)_\Gamma =(f,v)+(g,v)_\Gamma
\end{equation}
for all $v\in H^1(\Omega)$.
For \eqref{eq:froznonlinitcont}, find $\tilde w\in H^1(\Omega)$ such that
\begin{equation}\label{eq:auxpbyua}
\mathcal{A}(\Phi; \tilde w, v):=(\nabla \tilde w, \nabla v)-k^2((1+2\chi_D\varepsilon|\Phi|^2 )\tilde w, v)+ik(\tilde w,v)_\Gamma =(\tilde f,v)+(g,v)_\Gamma
\end{equation}
for all $v\in H^1(\Omega)$, where $\tilde{f}=f-k^2\varepsilon\chi_D |\Phi|^2\Phi$. Note that \eqref{eq:auxpb} and \eqref{eq:auxpbyua} are closely related because $\mathcal{A}(\Phi; w, v)=\mathcal{B}_\mathrm{lin}(\sqrt{2}\Phi; w, v)$ and the right-hand side is (slightly) different.
Consequently, we can deduce many results for \eqref{eq:auxpbyua} from their counterparts for \eqref{eq:auxpb}.
We emphasize that \eqref{eq:auxpb} is of Helmholtz-type, but with a variable, i.e., $x$-dependent, refractive index $n\coloneqq 1+\chi_D \varepsilon|\Phi|^2$ induced by $\Phi$.
Moreover, our assumptions on $\varepsilon$ and $D$ imply that $n$ is discontinuous over $\partial D$, i.e., the interface between the nonlinear material and the linear ``background''.

\cite{WuZ18} shows that there exists a constant $\theta_1$ such that if $k\varepsilon\|\Phi\|_{L^\infty(D)}^2\leq \theta_1$, a unique solution $w\in H^1(\Omega)$ to \eqref{eq:auxpb} exists and it satisfies the a priori estimates
\begin{equation}\label{eq:stabauxpb}
\|w\|_{1,k}\leq C_1 \Cdata,\qquad \|w\|_2\leq C_2 k\Cdata,\qquad \|w\|_{L^\infty(D)}\leq C_\infty  k^{(d-3)/2}\Cdata.
\end{equation}
In fact, these results on $w$ are the crucial ingredient to establish existence, uniqueness and a priori estimates (cf.~\eqref{eq:apriorinonlin}) for the solution $u$ to the nonlinear problem \eqref{eq:NLHweak} in \cite{WuZ18}.
By exploiting the correspondence between \eqref{eq:auxpb} and \eqref{eq:auxpbyua}, we directly obtain that, if $k\varepsilon\|\Phi\|_{L^\infty(D)}^2\leq \theta_1/2$, a unique solution $\tilde w\in H^1(\Omega)$ to \eqref{eq:auxpbyua} exists.
From \eqref{eq:stabauxpb} and \eqref{eq:auxpbyua}, we deduce the a priori estimates
\begin{equation}\label{eq:stabauxpbyua1}
\|\tilde w\|_{1,k}\leq C_1(\Cdata+k^2\varepsilon\|\Phi\|^2_{L^\infty(D)}\|\Phi\|_0)\leq C_1\Bigl(\Cdata+\frac{\theta_1}{2}\|\Phi\|_{1,k}\Bigr)
\end{equation}
as well as
\begin{equation}\label{eq:stabauxpbyua2}
\|\tilde w\|_2\leq C_2 k\Bigl(\Cdata+\frac{\theta_1}{2}\|\Phi\|_{1,k}\Bigr),\qquad \|\tilde w\|_{L^\infty(D)}\leq C_\infty k^{(d-3)/2}\Bigl(\Cdata+\frac{\theta_1}{2}\|\Phi\|_{1,k}\Bigr).
\end{equation}
These estimates will play a central role in the convergence proof for \eqref{eq:simplenewtonitcont} in Section \ref{subsec:conviter} below.

In the analysis of the finite element method for the linearized Helmholtz problem in Section \ref{subsec:analysislinearized}, we need a splitting of the continuous solution $w$ into an $H^2$-regular and an analytic part.
The recent results on solution splittings obtained by \cite{LaSpWu2021a,LaSpWu2021b} are not applicable since $n$ may exhibit a discontinuity across $\partial D$ in our case, see above.
Under a smallness of the data assumption, we show that the well-known solution splitting of the standard Helmholtz equation with $\varepsilon=0$ already yields the desired result for \eqref{eq:auxpb} as well. We note once more that the argument then transfers to \eqref{eq:auxpbyua} by the obvious modifications in the scaling of $\Phi$ and the form of $f$.

\begin{proposition}\label{prop:splitting}
	There is a constant $\theta_2$ such that, if $k\varepsilon\|\Phi\|_{L^\infty(D)}^2\leq \theta_2$, 
	the solution $w$ of \eqref{eq:auxpb} can be split as $w=w_{H^2}+w_{\mathcal{A}}$ with $w_{H^2}\in H^2(\Omega)$ and $w_{\mathcal{A}}$ analytic.
	Further, there exist $k$- and $\Phi$-independent constants $C, \gamma>0$ such that for any $m\in \mathbb{N}_0$
	\begin{align*}
	\|w_\mathcal{A}\|_{1,k}&\leq C\Cdata,\\
	\|\nabla^{m+2}w_\mathcal{A}\|_{0}&\leq C\gamma^m k^{-1} \max\{m, k\}^{m+2}\,\Cdata,\\
	\|w_{H^2}\|_2+k\|w_{H^2}\|_{1,k}&\leq C\Cdata.
	\end{align*}
\end{proposition}
The proof is presented in the appendix.

\subsection{Convergence of scheme \eqref{eq:simplenewtonitcont}}\label{subsec:conviter}
We now show that scheme \eqref{eq:simplenewtonitcont} satisfies a contraction property and therefore, the iteration sequence converges to the (unique) solution $u$ of \eqref{eq:NLHweak}.

\begin{proposition}\label{prop:convityua}
	Let $\{\uit{l}\}_{l\in \mathbb N_0}$ be defined via \eqref{eq:simplenewtonitcont} starting (for simplicity) from $\uit{0}\equiv 0$. 
	If  
	\begin{equation}\label{eq:smallnessyua}
	C_1C_\infty^2k^{d-2}\varepsilon\Cdata^2\leq q(1-q)^2
	\end{equation}
	for some $q< \min\{\frac{1}{6}, \frac{C_1\theta_1}{2}\}$ with $\theta_1$ introduced in Section \ref{subsec:linearized},
	we have for iteration scheme \eqref{eq:simplenewtonitcont} that
	\[\|\uit{l+1}-\uit{l}\|_{1,k}\leq \frac 12 \|\uit{l}-\uit{l-1}\|_{1,k}\]
	and the sequence $\{\uit{l}\}_l$ converges linearly to the solution $u$ of \eqref{eq:NLHweak}.
\end{proposition}

\begin{proof}
	\emph{First step: A priori estimates for $\uit{l}$.} We show by induction that for all $l\geq 1$, $\uit{l}$ is well-defined and satisfies
	\begin{align*}\|\uit{l}\|_{1,k}&\leq\frac{1}{(1-q)}C_1 \Cdata, \qquad\qquad \|\uit{l}\|_2\leq \frac{1}{(1-q)}C_2  k\Cdata,\\ \|\uit{l}\|_{L^\infty(D)}&\leq \frac{1}{(1-q)}C_\infty k^{(d-3)/2}\Cdata.
	\end{align*}
	
	The case $l=1$ directly follows from $\uit{0}=0$ and \eqref{eq:stabauxpbyua1}--\eqref{eq:stabauxpbyua2}.
	Let the statement be satisfied for $l$. Since $k\varepsilon\|\uit{l}\|_{L^\infty(D)}^2\leq C_\infty^2\frac{1}{(1-q)^2} k^{d-2}\varepsilon\Cdata^2\leq \theta_1/2$ by the assumptions, the discussion in Section \ref{subsec:linearized} yields that $\uit{l+1}$ is indeed well-defined. Moreover, we deduce from \eqref{eq:stabauxpbyua1} that
	\begin{align*}
	\|\uit{l+1}\|_{1,k}&\leq (C_1\Cdata+q\|\uit{l}\|_{1,k}),
	\end{align*}
	which recursively yields with $q<1$ that
	\begin{align*}
	\|\uit{l+1}\|_{1,k}\leq C_1\Cdata\sum_{j=0}^{l}q^j\leq \frac{1}{(1-q)}C_1\Cdata.
	\end{align*}
	Employing \eqref{eq:stabauxpbyua2}, we furthermore obtain
	\begin{align*}
	\|\uit{l+1}\|_2&\leq C_2 k\Bigl(\Cdata+\frac{\theta_1}{2}\|\uit{l}\|_{1,k}\Bigr)\leq C_2k\Cdata(1+\frac{q}{(1-q)})\\
	\|\uit{l+1}\|_{L^\infty(D)}&\leq C_\infty  k^{(d-3)/2}\Bigl(\Cdata+\frac{\theta_1}{2}\|\uit{l}\|_{1,k}\Bigr)\leq \frac{1}{1-q}C_\infty k^{(d-3)/2}\Cdata,
	\end{align*} 
	which finishes the first step.
	
	\emph{Second step: Contraction property.} Direct calculation shows that $\uit{l+1}-\uit{l}$ solves
	\[\mathcal{A}(\uit{l}; \uit{l+1}-\uit{l}, v)=k^2\varepsilon(|\uit{l-1}|^2(\uit{l-1}-\uit{l})+(|\uit{l}|^2-|\uit{l-1}|^2)\uit{l}, v)_D.\]
	The estimates from the first step yield
	\begin{align*}
	&\!\!\!\!\|\uit{l+1}-\uit{l}\|_{1,k}\\
	&\leq\frac{C_1}{(1-q)}k^2\varepsilon\Bigl( \|\uit{l-1}\|^2_{L^\infty(D)}\|\uit{l-1}-\uit{l}\|_0\\*
	&\qquad+\|\uit{l}\|_{L^\infty(D)}(\|\uit{l}\|_{L^\infty(D)}+\|\uit{l-1}\|_{L^\infty(D)})\|\uit{l}-\uit{l-1}\|_0\Bigr)\\
	&\leq 3\frac{C_1}{(1-q)^2}\varepsilon C_\infty^2k^{d-2}\Cdata^2\|\uit{l-1}-\uit{l}\|_{1,k}\leq 3q \|\uit{l-1}-\uit{l}\|_{1,k},
	\end{align*}
	where we used \eqref{eq:smallnessyua} in the last step. The assumption $q<1/6$ finishes the proof.
\end{proof}

The proposition explains the linear convergence observed in practice \cite{YuaL17}. However, the required smallness of the data assumption is more restrictive than for the frozen nonlinearity scheme. Precisely, following the proofs of \cite{WuZ18}, we see that the contraction property \eqref{eq:contractionfrozennonlin} holds if $C_1C_\infty^2 k^{d-2}\varepsilon\Cdata^2\leq \tilde \theta_0$ for $\tilde\theta_0<\min\{\frac{1}{4},\theta_1 C_1\}$, which is more relaxed in comparison to \eqref{eq:smallnessyua}.
Hence, Proposition \ref{prop:convityua} does not explain the better ``robustness'' of the scheme with respect to the data observed in \cite{YuaL17} as well as in our experiments in Section \ref{sec:numexp}.

\section{Nonlinear Helmholtz equation in the discrete setting}
\label{sec:nonlineardiscrete}
In this section, we turn to the finite element approximation of \eqref{eq:NLHweak}. We introduce the discretization using finite elements with higher-order polynomials in Section \ref{subsec:fem}. We then present the results of a priori error analysis, where we first consider the linearized problems in Section \ref{subsec:analysislinearized} and then the nonlinear problem in Section \ref{subsec:femnonlinear}. All proofs are collected in Section \ref{sec:proofs} and the appendix.

\subsection{Finite element discretization and notation}\label{subsec:fem}
Since we assume $\Gamma$ to be analytic, we will consider curved elements in order to have a conforming discretization. We follow the typical procedure as outlined in, e.g., \cite[Section 5]{MelS10}.
We assume that there exists a polyhedral/polygonal domain $\widetilde \Omega$ and a bi-Lipschitz mapping $\xi:\widetilde \Omega\to \Omega$. 
Let $\widetilde{\mathcal T}_h$ denote an admissible, shape regular simplicial mesh of $\widetilde \Omega$. We assume that the restrictions $\xi|_{\widetilde T}$ are analytic for all $\widetilde T\in \widetilde{\mathcal T}_h$.
We then set $\mathcal T_h=\{\xi(\widetilde T): T\in \widetilde{\mathcal T}_h\}$ as our mesh on $\Omega$ with mesh size $h\coloneqq \max_{T\in \mathcal T_h}\operatorname{diam} T$.
Note that for any $T=\xi(\widetilde T)\in \mathcal T_h$, there exists an affine, bijective mapping $A_T:\widehat T\to \widetilde T$ from the reference element $\widehat T$ (the unit simplex). Consequently, we have a mapping $F_T:\widehat T\to T$ via $F_T=R_T\circ A_T$ with $R_T=\xi|_{\widetilde T}$. We assume $F_T$, $R_T$ and $A_T$ to satisfy the smoothness and scaling assumptions of \cite[Assumption 5.2]{MelS10}.

For such a so-called quasi-uniform regular simplicial mesh $\mathcal T_h$, we denote the finite element space of piecewise (mapped) polynomials of degree $p$ by $V_{h,p}$, i.e.,
\[V_{h,p}\coloneqq \{v\in H^1(\Omega): v|_T\circ F_T\in \mathbb P_p(T)\text{ for all }T\in \mathcal T_h\},\]
where $\mathbb P_p$ denotes the polynomials of degree $p$.
We now seek the discrete solution $u_{h,p}\in V_{h,p}$ such that
\begin{equation}\label{eq:nhl-fem}
\mathcal{B}(u_{h,p}, v_h)=(f, v_h)+(g, v_h)_\Gamma 
\end{equation}
for all $v_h\in V_{h,p}$.
This yields a nonlinear system that we can solve via the discrete versions of the iteration schemes \eqref{eq:froznonlinitcont} or \eqref{eq:simplenewtonitcont}. As usual, these discrete versions are obtained by a Galerkin procedure, i.e., ansatz as well as test functions come from the space $V_{h,p}$.
As already done in the continuous case, we start the iterations with $\uhit{0}\equiv 0$ for simplicity.

\smallskip

We collect further finite element-related notation that will turn out useful in the error analysis.

Let $P_h:H^1(\Omega)\to V_{h,p}$ be the elliptic projection as defined by \cite{WuZ18} via
\begin{equation}\label{eq:defellproj}
(\nabla v_h, \nabla P_h \psi)+ik(v_h, P_h \psi)_\Gamma=(\nabla v_h, \nabla \psi)+ik(v_h, \psi)_\Gamma\qquad \text{for all }v_h\in V_{h,p}.
\end{equation}
This projection is well-defined and satisfies
\begin{equation}\label{eq:propellproj}
\|\psi-P_h\psi\|_0\lesssim h\|\psi-P_h \psi\|_{1,k}\lesssim  \inf_{v_h\in V_{h,p}}\|\psi-v_h\|_{1,k}.
\end{equation}
$P_h$ is related to the discrete Laplace operator $L_h:V_{h,p}\to V_{h,p}$ defined via
\begin{equation}\label{eq:Lh}
(L_h v_h, \psi_h)=(\nabla v_h, \nabla\psi_h)+ik(v_h, \psi_h)_\Gamma \qquad \text{for all }\psi_h\in V_{h,p}.
\end{equation}

Further, following \cite{DuWu2015}, we introduce discrete $H^j(\Omega)$-norms on $V_{h,p}$. We define the discrete operator $A_h:V_{h,p}\to V_{h,p}$ via
\begin{equation}\label{eq:Ah}
(A_h v_h, \psi_h):=(\nabla v_h, \nabla \psi_h)+(v_h, \psi_h).
\end{equation}
Let
\[0<\lambda_{1,h}<\lambda_{2,h}<\ldots \lambda_{\dim V_{h,p}, h}\]
denote its eigenvalues, which are all positive, and let $\varphi_{j,h}$ for $j=1,\ldots \dim V_{h,p}$ be the corresponding discrete eigenfunctions. For any real number $j$, the operator $A_h^j$ is defined via
\[A_h^j v_h=\sum_{l=1}^{\dim V_{h,p}}\lambda_{l,h}^j a_l\varphi_{l,h}\qquad \text{for }v_h=\sum_{l=1}^{\dim V_{h,p}} a_l\varphi_{l,h}.\]
The discrete norms on $V_{h,p}$ are then defined for any integer $j$ via
\begin{equation}\label{eq:defdiscretenorm}
\|v_h\|_{j,h}\coloneqq \|A_h^{j/2}v_h\|_0.
\end{equation}
For any $v_h\in V_{h,p}$, it holds that (see~\cite[Lemma 4.1, 4.2]{DuWu2015})
\begin{enumerate}
	\item for any integer $j$,
	\begin{equation}\label{eq:discretenorm-inverse}
	\|v_h\|_{j,h}\lesssim h^{-1}\|v_h\|_{j-1,h}
	\end{equation} 
	\item for any integer $0\leq j\leq p+1$,
	\begin{equation*}
	\|v_h\|_{-j, h}\lesssim \sum_{l=0}^{j}h^{j-l}\|v_h\|_{-l}.
	\end{equation*}
\end{enumerate}

\subsection{FEM error analysis for the auxiliary problem}\label{subsec:analysislinearized}
We can now analyze the linear auxiliary problem in the discrete setting. According to the discussion in Section \ref{subsec:linearized}, we focus on problem \eqref{eq:auxpbfem} below using $\mathcal{B}_\mathrm{lin}$, but note that everything carries over to the discrete version of \eqref{eq:auxpbyua} by the relation of $\mathcal{A}$ and $\mathcal{B}_\mathrm{lin}$.
Let $\Phi\in L^\infty(D)\cap H^1(\Omega)$ be given. Define $w_h\in V_{h,p}$ as the solution of
\begin{equation}\label{eq:auxpbfem}
\mathcal{B}_\mathrm{lin}(\Phi; w_h, v_h)=(f, v_h)+(g, v_h)_\Gamma
\end{equation}
for all $v_h\in V_{h,p}$.

\begin{lemma}\label{lem:erroraux}
	If $k\varepsilon\|\Phi\|_{L^\infty(D)}^2\leq \min\{\theta_1,\theta_2\}$ with the constants introduced in Section \ref{sec:setting}, there exists a constant $C_0>0$ such that, if $k(kh)^{2p}\leq C_0$, the finite element solution $w_h$ to the auxiliary problem \eqref{eq:auxpbfem} exists, is unique and satisfies
	\begin{equation}\label{eq:erraux}
	\|w-w_h\|_{1,k}\lesssim (1+k(kh)^p)\inf_{v_h\in V_{h,p}}\|w-v_h\|_{1,k}\lesssim (h+(kh)^p+k(kh)^{2p})\Cdata,
	\end{equation}
	where the constant in $\lesssim$ may depend on $p$.
	Further, $w_h$ is stable in the following sense
	\begin{equation}\label{eq:stabfemaux}
	\|w_h\|_{1,k}\lesssim \Cdata.
	\end{equation}
\end{lemma}
In the proof of the above lemma, we will also establish the following $L^2$-error estimate
\begin{equation}\label{eq:errL2aux}
\|w-w_h\|_0\lesssim (h^2+h(kh)^p +(kh)^{2p})\Cdata.
\end{equation}
Lemma \ref{lem:erroraux} essentially transfers \cite{DuWu2015} to a case where the coefficient $n$ in the Helmholtz problem is no longer constant. 
In contrast to the approach in \cite[Sec.~2.4]{Pem2020} and \cite{GaSp23}, we treat $n$ as a sufficiently small perturbation from the constant coefficient case. Therefore, some of our assumptions are different, in particular we can allow for lower regularity in $n$. Further, since we do not have a coefficient in the gradient part, i.e., $A=1$, we can use Robin boundary conditions everywhere, cf. the discussion in \cite[Rem.~2.62]{Pem2020}. Concerning the occurrence of $kh$ as first term in \eqref{eq:erraux}, we refer to the discussion after Theorem~\ref{thm:errornonlin}.
For convenience, we include the proof of Lemma~\ref{lem:erroraux} (along the lines of \cite{DuWu2015}) in the appendix.

Besides the stability of $w_h$ in the energy norm, we have the following $L^\infty$-estimate, which is important for the nonlinear case.

\begin{lemma}\label{lem:linfaux}
	If $k\varepsilon\|\Phi\|_{L^\infty(D)}^2\leq \min\{\theta_1,\theta_2\}$ with the constants introduced in Section \ref{sec:setting}, there exists a constant $C_1>0$ such that, if $k(kh)^{p+1}\leq C_1$, the solution $w_h$ of \eqref{eq:auxpbfem} satisfies
	\begin{equation}\label{eq:linfstabaux}
	\|w_h\|_{L^\infty(D)}\lesssim |\ln h|^{\overline p}\, k^{(d-3)/2}\Cdata,
	\end{equation}
	where $\overline p=1$ for $p=1$ and $\overline{p}=0$ for $p\geq 2$.
\end{lemma}
The proof follows the lines of \cite{WuZ18,JiLiWuZo2022} using the interior $L^\infty$ estimates of \cite{SchW1977}. The latter only introduce an $|\ln h|$-dependence in the case $p=1$. Note that \cite{JiLiWuZo2022} recently showed that even in the linear case, the $|\ln h|$ can be removed \emph{if} $d=2$.

\begin{remark}[Why the resolution conditions in Lemmas~\ref{lem:erroraux} and \ref{lem:linfaux} differ]\label{rem:resolcond}
	Comparing to the linear Helmholtz case, the resolution condition in Lemma \ref{lem:erroraux} is the known one  from \cite{DuWu2015}, while the  condition in Lemma \ref{lem:linfaux} corresponds to an older, sub-optimal condition in \cite{ZhWu2013}.
	To get the improved result similar to Lemma \ref{lem:erroraux}, \cite{DuWu2015} uses suitable discrete $H^j$-norms, cf.~\eqref{eq:defdiscretenorm}, and negative Sobolev norms in the estimation of $L^2$-scalar products (between a discrete and a projection error). In the $L^\infty$-norm estimate, however, we do not get such $L^2$-scalar products and therefore cannot use this technique. This is the main reason for the tighter resolution condition in Lemma \ref{lem:linfaux}. Note that both conditions agree for $p=1$, in particular the condition in Lemma \ref{lem:linfaux} agrees with the result for $p=1$ in \cite{WuZ18}.
\end{remark}

\subsection{Finite element method for the nonlinear problem}\label{subsec:femnonlinear}
We are now prepared to analyze the higher-order finite element method for the iteration schemes \eqref{eq:froznonlinitcont}  and \eqref{eq:simplenewtonitcont} applied to the nonlinear Helmholtz equation. 
We emphasize once more that the analysis in \cite{WuZ18,JiLiWuZo2022}, which partly inspires our proofs, is limited to $p=1$ and either iteration scheme \eqref{eq:froznonlinitcont} or Newton's method. 
We start with the convergence of the discrete schemes and, thereby, existence and uniqueness of the solution $u_{h,p}$ to \eqref{eq:nhl-fem}.

\begin{proposition}\label{prop:exfemnonlin}
	Let 
	\[k(kh)^{p+1}\leq C_1\]
	as in Lemma \ref{lem:linfaux}. Define 
	\[\sigma_j:=\tilde C_j|\ln h|^{2\overline p}\varepsilon k^{d-2}\Cdata^2,\]
	where $\tilde C_1, \tilde C_2$ associated with schemes \eqref{eq:froznonlinitcont} and \eqref{eq:simplenewtonitcont}, respectively, are some constants.
	If $\sigma_j<1$, the associated sequence $\{\uhit{l}\}_{l\in \mathbb N_0}\subset V_{h,p}$  starting at $\uhit{0}\equiv 0$ converges to the unique solution $u_{h,p}$ of \eqref{eq:nhl-fem} with rate $\sigma_j$, i.e.,
	\begin{equation}\label{eq:convit}
	\|u_{h,p}-\uhit{l}\|_{1,k}\lesssim \sigma_j^l\Cdata.
	\end{equation}
	Further, $u_{h,p}$ satisfies the stability estimates
	\begin{equation}\label{eq:stabuhp}
	\|u_{h,p}\|_{1,k}\lesssim \Cdata\quad\text{and}\quad\|u_{h,p}\|_{L^\infty(D)}\lesssim |\ln h|^{\overline p} k^{(d-3)/2}\Cdata.
	\end{equation}
\end{proposition}

We hence obtain the continuous as well as the discrete solution as limit of a sequence of solutions to linearized Helmholtz problems. 
The proposition gives us the convergence of the two iteration schemes also in the discrete setting as well as existence and uniqueness of the solution to \eqref{eq:nhl-fem}. Of course, this unique solution exists as soon as $\sigma_j<1$ for $j=1$ \emph{or} $j=2$. In other words, we can choose the less restrictive condition when we consider properties of the discrete solution to the nonlinear problem.
Using the error estimates in the linear case (cf. Lemma \ref{lem:erroraux}), we can conclude our main result on the finite element error.

\begin{theorem}\label{thm:errornonlin}
	If $k(kh)^{p+1}\leq C_1$ and $|\ln h|^{2\overline p}\varepsilon k^{d-2}\Cdata^2\leq \theta$ sufficiently small, the unique finite element solution $u_{h,p}$ to \eqref{eq:nhl-fem} satisfies the error estimate
	\[\|u-u_{h,p}\|_{1,k}\lesssim (h+(kh)^p+k(kh)^{2p})\Cdata.\]
\end{theorem}
Note that by combining the previous theorem and \eqref{eq:convit} we deduce an error estimate for $u-\uhit{l}$ and any of the two schemes \eqref{eq:froznonlinitcont} or \eqref{eq:simplenewtonitcont} by the triangle inequality.

Theorem \ref{thm:errornonlin} bounds the error between the exact and the discrete solution of the nonlinear Helmholtz equation, where the dependence on $k$, $h$, and $p$ is the same as in the linear case.
We provide a pre-asymptotic error bound with the so-called pollution term $k(kh)^{2p}$ under a resolution condition discussed further in Remark \ref{rem:resolcond}.
Note that the first term $h$ occurs in Theorem \ref{thm:errornonlin} since we do not assume more than $H^2(\Omega)$-regularity of $u$. Nevertheless, this term can be interpreted as a higher order term in the sense that it does not dictate the rate of convergence due to the resolution condition, cf.~\cite[Rem.~2.40]{Pem2020}.
As in \cite{DuWu2015} for the linear Helmholtz equation, our result assumes a fixed polynomial degree in the sense that the involved constants depend on $p$. We strongly believe that the famous $hp$-error analysis in the asymptotic regime \cite{MelS10,MelS11} can be transferred to the nonlinear Helmholtz equation for sufficiently small data as well.
An $hp$-version of the result in Lemma \ref{lem:erroraux} in the asymptotic regime is already available, see \cite{LaSpWu2021a}, but to the best of our knowledge, nothing is known about an $hp$-version of the $L^\infty$-estimate in Lemma \ref{lem:linfaux}. One can circumvent the application of Lemma \ref{lem:linfaux} as in \cite{MaiV22}, but the price to pay is a worse $k$-dependence in the smallness of the data assumption.

\section{Proofs of the results in Section \ref{sec:nonlineardiscrete}}\label{sec:proofs}
In this section, we prove our main results Lemma \ref{lem:linfaux}, Proposition \ref{prop:exfemnonlin}, and Theorem \ref{thm:errornonlin}.

\subsection{Proof of Lemma \ref{lem:linfaux}}

\begin{proof}
	We set $\tilde w_h=\overline{P_h \overline{w}}$ with $P_h$ defined in \eqref{eq:defellproj} and $\eta_h=\tilde w_h-w_h$. 
	Further, we introduce $\eta\in H^1(\Omega)$ as the solution of
	\[(\nabla  \eta, \nabla v)+ik(\eta, v)_\Gamma=k^2((1+\chi_D \varepsilon |\Phi|^2)(w-w_h), v)\qquad \text{for all }v\in H^1(\Omega).\]
	By the triangle inequality, it holds
	\begin{align*}
	\|w_h\|_{L^\infty(D)}&\leq \|\eta\|_{L^\infty(D)}+\|\eta-\eta_h\|_{L^\infty(D)}+\|w-\tilde w_h\|_{L^\infty(D)}+\|w\|_{L^\infty(D)}\\&=:T_1+T_2+T_3+T_4.
	\end{align*}
	
	\emph{First step: Estimate of $T_1$:} 
	By elliptic regularity theory and \eqref{eq:errL2aux} we deduce
	\begin{equation}\label{eq:H2eta}
	\|\eta\|_2\lesssim k^2\|w_h-w\|_0\lesssim (k^2h^2+k(kh)^{p+1}+k^2(kh)^{2p})\Cdata.
	\end{equation}
	We observe that $\eta_h$ satisfies
	\[(L_h \eta_h, v_h)=k^2((1+\chi_D \varepsilon |\Phi|^2)(w-w_h), v_h)\qquad \text{for all }v_h\in V_{h,p}\]
	with $L_h$ as defined in \eqref{eq:Lh}. 
	Hence, $\eta_h$ is the finite element approximation of $\eta$.  
	Standard finite element theory then yields
	\begin{equation}\label{eq:femerreta}
	\begin{split}
	\|\eta-\eta_h\|_0&\lesssim h^2\|\eta\|_2\lesssim h^2k^2\|(1+\chi_D \varepsilon |\Phi|^2)(w-w_h)\|_0\\
	&\lesssim (k^2h^4+h(hk)^{p+2}+(kh)^{2p+2})\Cdata,
	\end{split}
	\end{equation}
	where we used \eqref{eq:errL2aux}.
	Next, we re-write the equation for $\eta$ and observe that it solves
	\begin{align*}
	(\nabla \eta, \nabla v)-k^2(\eta, v)+ik(\eta,v)_\Gamma = k^2(\eta_h-\eta, v)+k^2(w-\tilde w_h, v)+k^2(\varepsilon|\Phi|^2( w -w_h), v)_D.
	\end{align*}
	Hence, we obtain by~\eqref{eq:stabauxpb} together with \eqref{eq:femerreta} and \eqref{eq:errL2aux} that 
	\begin{align*}
	\|\eta\|_{L^\infty(D)}&\lesssim k^{(d-3)/2}k^2(\|\eta_h-\eta\|_0+\|\tilde w_h-w\|_0+k^{-1}\theta_1\|w_h-w\|_0)\\
	&\lesssim k^{(d-3)/2}k^2\bigl( k^2h^4+h(kh)^{p+2}+(kh)^{2p+2}+h^2+(kh)^ph+k^{-1} h^2\\
	&\qquad\qquad\qquad+k^{-1}h(kh)^p +k^{-1}(kh)^{2p}\bigr)\Cdata\\
	&=k^{(d-3)/2}\bigl(kh^2+(kh)^2+(kh)^4+(kh)^{p+1}+k(kh)^{p+1}\\
	&\qquad\qquad\quad+k(kh)^{p+3}+k^2(kh)^{2p+2}+k(kh)^{2p}\bigr)\Cdata\\
	&\lesssim k^{(d-3)/2}\Cdata,
	\end{align*}
	where we used the resolution condition $k(kh)^{p+1}\leq C_1$ in the last step.
	
	\emph{Second step: Estimate of $T_2$:}
	Note that $H^1(\Omega)$ can be continuously embedded into $L^6(\Omega)$ for $d\leq 3$.
	Take a subdomain $D_1$ with $D\subset D_1$ and $\operatorname{dist}(\partial D, \partial D_1)\approx \operatorname{dist}(\partial D_1, \partial \Omega)\approx 1$. Interior $L^\infty$-error estimates \cite[Thm.~5.1]{SchW1977}, interpolation estimates, interior Schauder estimates for elliptic equations \cite[Thm. 9.11]{GiTr2001}, and the $k$-weighted Nierenberg inequality from \cite[Lemma 2.3]{MaiV22} imply
	\begin{align*}
	\|\eta-\eta_h\|_{L^\infty(D)}&\lesssim |\ln h|^{\overline p}\, \|\eta-I_h\eta\|_{L^\infty(D_1)}+\|\eta-\eta_h\|_0\\
	&\lesssim  |\ln h|^{\overline p}\, h^{2-d/6} \|\eta\|_{W^{2,6}(D_1)}+\|\eta-\eta_h\|_0\\
	&\lesssim  |\ln h|^{\overline p}\, h^{2-d/6} (\|\eta\|_{L^6(\Omega)}+k^2\|w_h-w\|_{L^6(\Omega)})+\|\eta-\eta_h\|_0\\*
	&\lesssim  |\ln h|^{\overline p}\, h^{2-d/6}(\|\eta\|_2+k^{1+d/3}\|w-w_h\|_{1,k})+\|\eta-\eta_h\|_0,
	\end{align*}
	where $I_h$ denotes the nodal interpolation operator.
	In the first step, we already showed -- using \eqref{eq:femerreta} -- that $k^2\|\eta-\eta_h\|_0$ is uniformly bounded under the resolution condition $k(kh)^{p+1}\leq C_1$. Thereby we easily deduce that under the same resolution condition, $\|\eta-\eta_h\|_0\lesssim k^{-2}\lesssim k^{(d-3)/2}$.
	Hence, we only need to bound $h^{2-d/6}(\|\eta\|_2+k^{1+d/3}\|w-w_h\|_{1,k})$ in the following.
	With \eqref{eq:H2eta}, we obtain
	\begin{align*}
	h^{2-d/6}(\|\eta\|_2+k^{1+d/3}\|w-w_h\|_{1,k})&\lesssim h^{2-d/6}(k^2\|w-w_h\|_0+k^{1+d/3}\|w-w_h\|_{1,k})\\
	&\lesssim h^{2-d/6}k^{1+d/3}\|w-w_h\|_{1,k}.
	\end{align*}
	Applying \eqref{eq:erraux}, we deduce
	\begin{align*}
	h^{2-d/6}k^{1+d/3}\|w-w_h\|_{1,k}&\lesssim h^{2-d/6}k^{1+d/3}(h+(kh)^p+k(kh)^{2p})\Cdata\\
	&\lesssim k^{(d-3)/2}\bigl( (kh)^{2-d/6}k^{1/2}(h+(kh)^p+k(kh)^{2p})\bigr)\Cdata\\
	&\lesssim k^{(d-3)/2}\Cdata, 
	\end{align*}
	where we used the resolution condition $k(kh)^{p+1}\leq C_1$ and $k\gtrsim 1$ in the last inequality.
	Combining all estimates in this step, we showed
	\[\|\eta-\eta_h\|_{L^\infty(D)}\lesssim |\ln h|^{\overline p}\, k^{(d-3)/2}\Cdata\]
	under the resolution condition $k(kh)^p\leq C_1$.
	
	\emph{Third step: Estimate of $T_3$:}	
	Take a subdomain $D_1$ with $D\subset D_1$ and $\operatorname{dist}(\partial D, \partial D_1)\approx \operatorname{dist}(\partial D_1, \partial \Omega)\approx 1$ as in the previous step.
	We obtain with \cite[Thm.~5.1]{SchW1977}, \eqref{eq:propellproj} and \eqref{eq:stabauxpb}
	\begin{align*}
	\|w-\tilde w_h\|_{L^\infty(D)}&\lesssim |\ln h|^{\overline p}\,\|w\|_{L^\infty(D_1)}+\|w-\tilde w_h\|_0\\
	&\lesssim |\ln h|^{\overline p}\,\|w\|_{L^\infty(D_1)}+h\|w\|_{1,k}\\
	&\lesssim \Bigl(|\ln h|^{\overline p}\, k^{(d-3)/2}+h\Bigr) \Cdata
	\lesssim |\ln h|^{\overline p}\,k^{(d-3)/2}\Cdata,
	\end{align*}
	where we used $kh\lesssim 1$ and $k\gtrsim 1$ in the last step.
	
	Combining steps 1-3 with \eqref{eq:linfstabaux} for $T_4$ yields the assertion.
\end{proof}

\subsection{Proofs from Section \ref{subsec:femnonlinear}}

\begin{proof}[Proof of Proposition \ref{prop:exfemnonlin}]
	\emph{First step: Iteration scheme \eqref{eq:froznonlinitcont}:} As before, let for simplicity $\uhit{0}\equiv0$.
	From Lemma \ref{lem:erroraux} we obtain that a unique solution $\uhit{1}\in V_{h,p}$ to the discrete version of \eqref{eq:froznonlinitcont} exists and moreover, that it satisfies due to \eqref{eq:stabfemaux} and Lemma \ref{lem:linfaux}
	\[\|\uhit{1}\|_{1,k}\lesssim\Cdata\quad \text{and}\quad \|\uhit{1}\|_{L^\infty(D)}\lesssim |\ln h |^{\overline p}\,k^{(d-3)/2}\Cdata.\]
	The latter implies that $k\varepsilon\|\uhit{1}\|_{L^\infty(D)}^2\lesssim \min\{\theta_1,\theta_2\}$ so that the assumptions of Lemma \ref{lem:erroraux} are satisfied.
	Inductively, we conclude that the whole iteration sequence \eqref{eq:froznonlinitcont} exists. Each $\uhit{j}$ is the unique solution of a linearized Helmholtz problem and satisfies the stability estimates
	\[\|\uhit{l}\|_{1,k}\lesssim \Cdata\quad \text{and}\quad \|\uhit{l}\|_{L^\infty(D)}\lesssim |\ln h |^{\overline p}\,k^{(d-3)/2}\Cdata\]
	with constants in $\lesssim$ independent of $l$.
	
	Set $v_h^{(l)}=\uhit{l+1}-\uhit{l}$ and observe that $v_h^{(l)}$ solves
	\[\mathcal B_{\mathrm{lin}}(\uhit{l}; v_h^{(l)}, \psi_h)=k^2(\varepsilon \uhit{l}(|\uhit{l}|^2-|\uhit{l-1}|^2), \psi_h)_D\qquad \text{for all}\quad \psi_h\in V_{h,p}.\]
	As we have $k\varepsilon\|\uhit{l}\|_{L^\infty}\lesssim \min\{\theta_1,\theta_2\}$, we can again apply Lemma \ref{lem:erroraux} to obtain
	\begin{align*}
	\|v_h^{(l)}\|_{1,k}&\lesssim k^2\varepsilon\|\uhit{l}(|\uhit{l}|^2-|\uhit{l-1}|^2)\|_{0,D}\\
	&\lesssim k\varepsilon\|\uhit{l}\|_{L^\infty(D)}\Bigl(\|\uhit{l-1}\|_{L^\infty(D)}+\|\uhit{l}\|_{L^\infty(D)}\Bigr)\|v_h^{(l)}\|_{1,k}\\
	&\lesssim |\ln h|^{2\overline p}\,\varepsilon k^{d-2}\Cdata^2\|v_h^{(l)}\|_{1,k}.
	\end{align*}
	Hence if $\sigma_1:=\tilde C_1|\ln h|^{2\overline p}\varepsilon k^{d-2}\Cdata^2<1$, $\{v_h^{(l)}\}_{l\in \mathbb N_0}$ forms a strictly contracting sequence or, in other words, the iterations $\{\uhit{l}\}_{l\in \mathbb N_0}$ form a Cauchy sequence. Therefore, they converge to some $u_{h,p}$ and it is easy to verify that $u_{h,p}$ is a solution to \eqref{eq:nhl-fem}.
	Further, $u_{h,p}$ satisfies the stability estimates \eqref{eq:stabuhp} as they are satisfied in all iterations (see above).
	Following the arguments in the estimate for $v_h^{(l)}$ above, we can also conclude the uniqueness of $u_{h,p}$.
	Finally, \eqref{eq:convit} for $j=1$ follows as in the Banach fixed-point theorem.
	
	\emph{Second step: Iteration scheme \eqref{eq:simplenewtonitcont}:} We transfer the proof of Proposition \ref{prop:convityua} to the discrete setting.
	We show by induction that for all $l\geq 1$, $\uhit{l}$ is well-defined and satisfies
	\[\|\uhit{l}\|_{1,k}\lesssim  \Cdata, \quad \|\uhit{l}\|_{L^\infty(D)}\lesssim |\ln h|^{\overline p}\, k^{(d-3)/2}\Cdata,\]
	where the constants in $\lesssim$ may depend on $\sigma_2$, but not on $l$.
	
	The case $l=1$ directly follows from $\uhit{0}=0$ and \eqref{eq:stabfemaux} and Lemma \ref{lem:linfaux}.
	Let the statement be satisfied for $l$. Since $k\varepsilon\|\uhit{l}\|_{L^\infty(D)}^2\lesssim |\ln h|^{2\overline p}k^{d-2}\varepsilon\Cdata^2\lesssim \min\{\theta_1, \theta_2\}$ by assumption, we can deduce from Lemma \ref{lem:erroraux} that $\uhit{l+1}$ is indeed well-defined, see also the discussion in Section \ref{subsec:linearized}. Moreover, we have
	\begin{align*}
	\|\uhit{l+1}\|_{1,k}&\lesssim (\Cdata+\sigma_2\|\uit{l}\|_{1,k})
	\end{align*}
	With the assumption $\sigma_2<1$ we obtain recursively
	\begin{align*}
	\|\uhit{l+1}\|_{1,k}\lesssim \Cdata\sum_{j=0}^{l}\sigma_2^j\lesssim \Cdata.
	\end{align*}
	Employing Lemma \ref{lem:linfaux}, we furthermore obtain
	\begin{align*}
	\|\uhit{l+1}\|_{L^\infty(D)}&\lesssim  |\ln h|^{\overline p}\,k^{(d-3)/2}(\Cdata+\sigma_2\|\uhit{l}\|_{1,k})\lesssim |\ln h|^{\overline p}\,k^{(d-3)/2}\Cdata.
	\end{align*} 
	As in the first step, we define $v_h^{(l)}:=\uhit{l+1}-\uhit{l}$, which solves
	\[\mathcal{A}(\uhit{l}; v_h^{(l)}, \psi_h)=k^2\varepsilon(-|\uhit{l-1}|^2v_h^{(l-1)}+(|\uhit{l}|^2-|\uhit{l-1}|^2)\uhit{l}, \psi_h)_D.\]
	The a priori estimates for $\uhit{l}$ from above yield with \eqref{eq:stabfemaux}
	\begin{align*}
	\|v_h^{(l)}\|_{1,k}&\lesssim k^2\varepsilon\Bigl( \|\uhit{l-1}\|^2_{L^\infty(D)}\|v_h^{(l-1)}\|_0\\
	&\qquad\qquad +\|\uhit{l}\|_{L^\infty(D)}(\|\uhit{l}\|_{L^\infty(D)}+\|\uhit{l-1}\|_{L^\infty(D)})\|v_h^{(l-1)}\|_0\Bigr)\\
	&\lesssim k\varepsilon |\ln h|^{2\overline p}\,k^{d-3}\Cdata^2\|v_h^{(l-1)}\|_{1,k}.
	\end{align*}
	Hence if $\sigma_2:=C_2|\ln h|^{2\overline p}\varepsilon k^{d-2}\Cdata^2<1$, $\{v_h^{(l)}\}_{l\in \mathbb N_0}$ forms a strictly contracting sequence or, in other words, the iterations $\{\uhit{l}\}_{l\in \mathbb N_0}$ form a Cauchy sequence. Therefore, they converge to some $u_{h,p}$, which one can easily identify as the unique solution to \eqref{eq:nhl-fem} from the first step. The stability estimates \eqref{eq:stabuhp} are then already known, and, finally, \eqref{eq:convit} for $j=2$  again follows as in the Banach fixed-point theorem.
\end{proof}

\begin{proof}[Proof of Theorem \ref{thm:errornonlin}]
	We proceed as in \cite{WuZ18}. Let $\{\uit{l}\}_{l\in \mathbb N_0}$ and $\{\uhit{l}\}_{l\in \mathbb N_0}$ be iteration sequences for $u$ and $u_{h,p}$, respectively.
	We bound the error $\uit{l}-\uhit{l}$ and at the very end, let $l\to\infty$. Since we eventually consider the limit, we simply work with the iteration sequences defined according to \eqref{eq:froznonlinitcont}. The adaption to \eqref{eq:simplenewtonitcont} is straightforward and yields a similar result.
	We define a sequence $\{\tilde u_h^{(l)}\}_{l\in \mathbb N_0}$ via $\tilde u_h^{(0)}=\uhit{0}$ and $\tilde u_h^{(l)}\in V_{h,p}$ solves
	\[\mathcal{B}_{\mathrm{lin}}(\uit{l-1}; \tilde u_h^{(l)}, v_h)=(f, v_h)+(g, v_h)_\Gamma\qquad \text{for all}\quad v_h\in V_{h,p}.\]
	We split the error as $\uit{l}-\uhit{l}=(\uit{l}-\tilde u_h^{(l)})+(\tilde u_h^{(l)}-\uhit{l})$ and estimate both terms separately.
	For the first term, we obtain directly from Lemma \ref{lem:erroraux} that
	\begin{equation}\label{eq:errornlh-proof}
	\|\uit{l}-\tilde u_h^{(l)}\|_{1,k}\lesssim (h+(kh)^p+k(kh)^{2p}) \Cdata.
	\end{equation}
	It remains to estimate $\eta_h^{(l)}\coloneqq \tilde u_h^{(l)}-\uhit{l}$. We observe that $\eta_h^{(l)}\in V_{h,p}$ solves
	\[\mathcal B_{\mathrm{lin}}(\uit{l-1}; \eta_h^{(l)}, v_h)=k^2(\varepsilon(|\uit{l-1}|^2-|\uhit{l-1}|^2)\uhit{l}, v_h)_D\qquad \text{for all}\quad v_h\in V_{h,p}.\]
	From Lemmas \ref{lem:erroraux} and \ref{lem:linfaux}, we hence obtain
	\begin{align*}
	\|\eta_h^{(l)}\|_{1,k}&\lesssim k^2\varepsilon\|(|\uit{l-1}|^2-|\uhit{j-1}|^2)\uhit{l}\|_{0,D}\\
	&\lesssim k^2\varepsilon\|\uhit{l}\|_{L^\infty(D)}\bigl(\|\uit{l-1}\|_{L^\infty(D)}+\|\uhit{l-1}\|_{L^\infty(D)}\bigr) \|\uit{l-1}-\uhit{l-1}\|_{0,D}\\
	&\lesssim k^2\varepsilon\bigl(|\ln h|^{\overline p}\, k^{(d-3)/2}\Cdata\bigr)^2\|\uit{l-1}-\uhit{l-1}\|_{0,D}\\
	&\lesssim  |\ln h|^{2\overline p}\,k^{d-2}\Cdata^2\bigl(\|\uit{l-1}-\tilde u_h^{(l-1)}\|_{1,k}+\|\eta_h^{(l-1)}\|_{1,k}\bigr).
	\end{align*}
	If $|\ln h|^{2\overline p}k^{d-2}\varepsilon\Cdata^2$ is sufficiently small, we consequently have
	\[\|\eta_h^{(l)}\|_{1,k}\leq \frac 12\|\uit{l-1}-\tilde u_h^{(l-1)}\|_{1,k}+\frac 12\|\eta_h^{(l-1)}\|_{1,k}.\]
	By induction together with \eqref{eq:errornlh-proof} and $\eta_h^{(0)}=0$, we deduce
	\begin{align*}
	\|\eta_h^{(l)}\|_{1,k}&\lesssim \sum_{j=0}^{l-1}2^{l-j}\|\uit{j}-\tilde u_h^{(j)}\|_{1,k}\\
	&\lesssim (h+(kh)^p+k(kh)^{2p})\Cdata+2^{-l}\|\uit{0}-\uhit{0}\|_{1,k}.
	\end{align*}
	Finally, employing the triangle inequality, \eqref{eq:errornlh-proof} and $\uit{0}=\uhit{0}\equiv 0$ yields
	\begin{align*}
	\|\uit{l}-\uhit{l}\|_{1,k}\lesssim (h+(kh)^p+k(kh)^{2p})\Cdata.
	\end{align*}
	Letting $j\to \infty$ finishes the proof.
\end{proof}

\section{Numerical experiments}\label{sec:numexp}
In this section, we investigate the higher-order FEM for the nonlinear Helmholtz problem numerically. We first focus on the discretization error in dependence on the mesh size $h$, the polynomial degree $p$ and the wave number $k$, where we aim to illustrate the theoretical findings of Theorem \ref{thm:errornonlin}.
In the second part of the numerical experiments, we examine the convergence of the nonlinear iteration, where we aim to analyze the dependence of the contraction factor on $k,h,p, \varepsilon$ and the data, cf.~\eqref{eq:convit}. We also compare the two different iteration schemes presented in Section \ref{subsec:iter}.
For all experiments, we set $\Omega=B_1(0)\subset \mathbb{R}^2$ and $D=B_{0.5}(0)$. 
All simulations\footnote{The code to reproduce the results is available at Zenodo under DOI \texttt{10.5281/zenodo.7016963}.} were obtained with NGSolve \cite{Sch1997,Sch2014}.

\subsection{Convergence of the discretization error}
As data, we choose $g\equiv 0$ and 
\[f=\begin{cases}
10000\exp\Bigl(-\frac{1}{1.2-\bigl(\frac{|x-x_0|}{0.05}\bigr)^2}\Bigr)&\text{if }\frac{|x-x_0|}{0.05}<1,\\
0&\text{else},
\end{cases}\] 
with $x_0=(-0.55, 0)$.
Since an analytical solution is not known, we use the finite element solution on a mesh with $h=2^{-7}$ and polynomial degree $p=3$ as reference solution for the following error plots. 
We always depict relative errors in the $\|\cdot\|_{1,k}$ norm.
Unless otherwise mentioned, we solve the nonlinear system using the frozen nonlinearity iteration \eqref{eq:froznonlinitcont} until either the (relative) residual is smaller than $5 \cdot 10^{-7}$ or the maximum of $20$ iterations is reached.

\begin{figure}
	\includegraphics[trim=10mm 5mm 15mm 15mm, clip=true, width=0.31\textwidth]{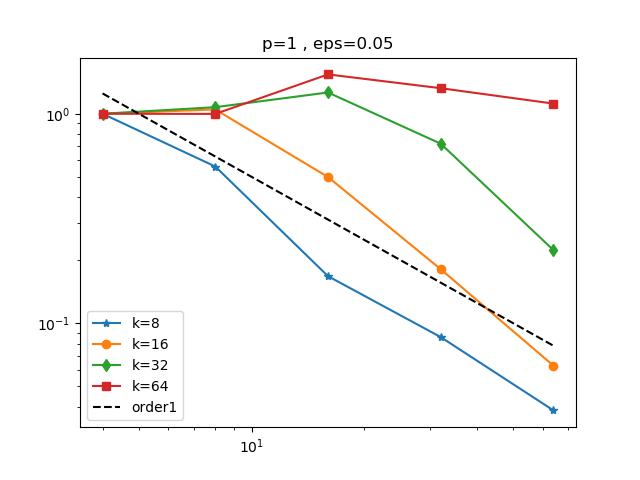}%
	\hspace*{2ex}%
	\includegraphics[trim=10mm 5mm 15mm 15mm, clip=true, width=0.31\textwidth]{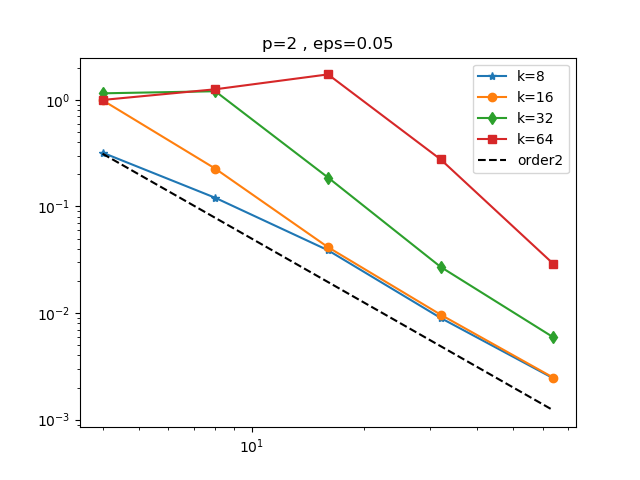}%
	\hspace*{2ex}%
	\includegraphics[trim=10mm 5mm 15mm 15mm, clip=true, width=0.31\textwidth]{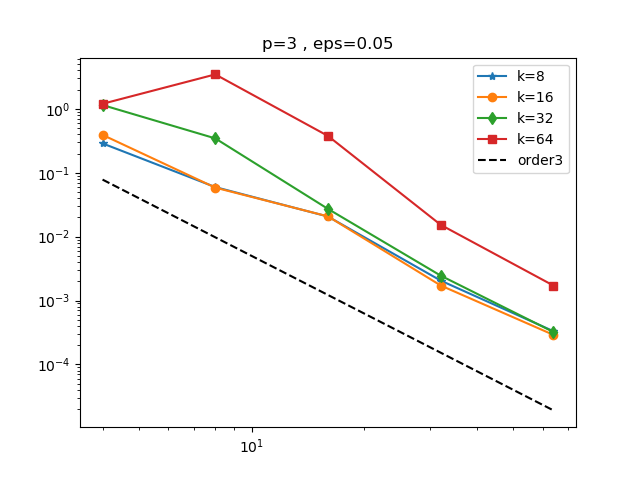}\\
	\includegraphics[trim=10mm 5mm 15mm 15mm, clip=true, width=0.31\textwidth]{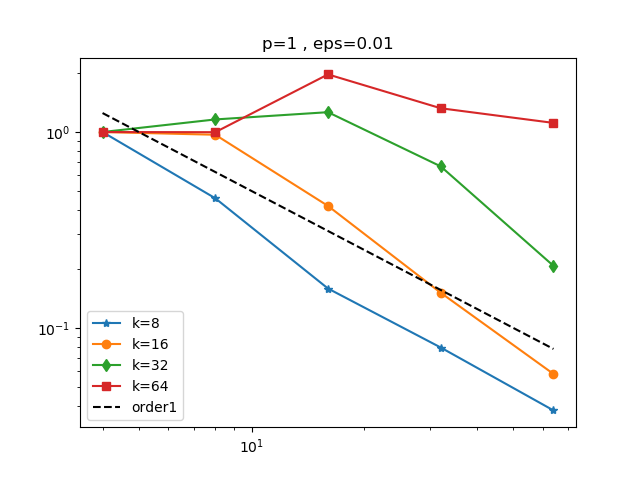}%
	\hspace*{2ex}%
	\includegraphics[trim=10mm 5mm 15mm 15mm, clip=true, width=0.31\textwidth]{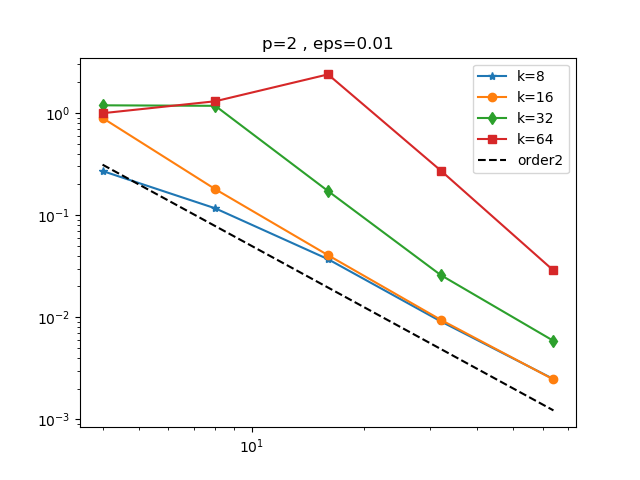}%
	\hspace*{2ex}%
	\includegraphics[trim=10mm 5mm 15mm 15mm, clip=true, width=0.31\textwidth]{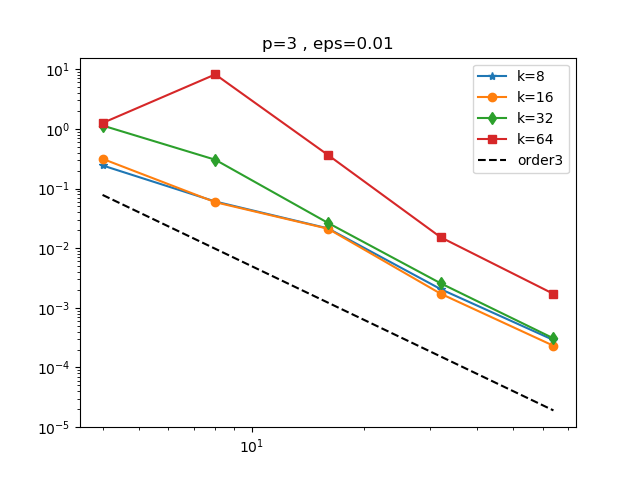}
	\caption{Relative error in the energy norm versus mesh size for the first experiment. Top row $\varepsilon=0.05$, bottom row $\varepsilon=0.01$. In each row from left to right $p=1,2,3$.}
	\label{fig:exp1hconvergence}
\end{figure}

The results for two different values of $\varepsilon$ are depicted in Figure \ref{fig:exp1hconvergence}. Firstly, we observe that the (asymptotic) error behavior is not influenced by $\varepsilon$ as expected by Theorem \ref{thm:errornonlin}.
Moreover, we confirm the expected convergence rates $h^p$ for this smooth right-hand side. Similar to and as expected from the linear case, we further observe that the plateau of error stagnation is larger for growing wave numbers, but this pollution effect can be reduced by increasing the polynomial degree.
We emphasize that we could always achieve a relative residual of at least $5 \cdot 10^{-7}$ in less than $20$ iterations in the convergence regime. On the other hand, for very coarse meshes (when the condition $k(kh)^{2p}\lesssim 1$ is not met), the fixed-point iteration may not converge. We stress that this does not contradict our theory.

To compare the error convergence for different $p$ from another perspective, we investigate how many degrees of freedom are required to obtain a relative energy error below a certain tolerance, say, $0.06$.
As Figure~\ref{fig:exp1hconvergence} suggests that the behavior for the two different $\varepsilon$ values is quite similar, we fix $\varepsilon=0.01$ in the following.
Table~\ref{tab:dofs} summarizes our findings, where $--$ indicates that we could not achieve the desired energy error with the considered meshes.
We make two important observations. First, for fixed wave number, the required number of degrees of freedom for the targeted accuracy decreases when we increase the polynomial degree. This is explained by the better convergence rate and the shorter stagnation  phase.
Second, for the wave number $k=8,16,32$, we can obtain the desired accuracy with (at most) $7,746$ degrees of freedom. For this, we need to slightly increase the polynomial degree, which is especially visible if the results for $k=16$ and $k=32$ are compared. 
Note that we expect to reduce the required number of degrees of freedom for $k=64$ if we considered even higher order spaces with $p\geq 4$. 
These observations agree very well with the $hp$-FEM convergence analysis of the linear case where the polynomial degree should be adapted like $p\gtrsim \log k$, see \cite{MelS10,MelS11}.

\begin{table}
	\centering
	\begin{tabular}{l|c|c|c|c}
		&$k=8$&$k=16$&$k=32$&$k=64$\\[0.5ex]\hline
		$p=1$&$14,560$&$14,560$&$--$&$--$\\
		&\footnotesize $(0.0379)$&\footnotesize $(0.0582)$&&\\[0.5ex]\hline
		$p=2$&$3,481$&$3,481$&$14,435$&$57,801$\\
		&\footnotesize $(0.0370)$&\footnotesize $(0.0405)$&\footnotesize $(0.0259)$&\footnotesize $(0.0287)$\\[0.5ex]\hline
		$p=3$&$7,746$&$1,788$&$7,746$&$32,358$\\
		&\footnotesize $(0.0214)$&\footnotesize $(0.0590)$&\footnotesize $(0.0265)$&\footnotesize $(0.0152)$\\[0.5ex]\hline
	\end{tabular}
	\caption{Required numbers of degrees of freedom to reach a relative energy error below $0.06$. In brackets the attained relative energy error is given.}
	\label{tab:dofs}
\end{table}

\subsection{Convergence of the nonlinear iteration}
We now turn to the behavior of the iteration schemes and we aim to shed light on their dependence on $k, h, p,\varepsilon$ and the data.
In our investigations, we will study (i) the number of iterations required to reach a (relative) residual of $5\cdot 10^{-7}$ or (ii) the contraction factor
\[\sigma^{(l)}\coloneqq \frac{\|\uhit{l}-\uhit{l-1}\|_{1,k}}{\|\uhit{l-1}-\uhit{l-2}\|_{1,k}}\qquad l\geq 2\]
in the $l$-th iteration step.

\begin{table}
	\centering
	\begin{tabular}{l|c|c||c|c||c|c}
		$(h,p)$&$2^{-4},2$&$2^{-4},5$&$2^{-4}, 1$&$2^{-7}, 1$&$2^{-3}, 2$& $2^{-6}, 2$\\\hline
		it.& $15$&$15$&$15$&$18$&$15$&$16$\\
		res.&$4.966\cdot 10^{-7}$&$4.980\cdot 10^{-7}$&$2.472\cdot 10^{-7}$&$2.960\cdot 10^{-7}$&$3.443\cdot 10^{-7}$&$2.671\cdot 10^{-7}$
	\end{tabular}
	\caption{Number of ``frozen nonlinearity'' iterations and final residuals for different values of $h$ and $p$.}
	\label{tab:frozenit-hp}
\end{table}

First, we fix $k=8$, $\varepsilon =0.1$, $g\equiv 0$ and $f\equiv 50$ and investigate the $h$ and $p$-dependence. We give the number of iterations till convergence (as explained above) for the iteration scheme \eqref{eq:froznonlinitcont} in Table~\ref{tab:frozenit-hp}. In the second and third column we see that the iteration scheme does not seem to be affected by $p$. This is in good agreement with \eqref{eq:convit} in Proposition \ref{prop:exfemnonlin}, where the contraction factor does not contain $p$. In the other columns we compare the dependence on $h$ for $p=1$ and $p=2$. Each time we refine the coarse mesh three times. The number of iterations only increases by a few steps. Since \cite{JiLiWuZo2022} shows that the $|\ln h|$ can be removed in two space dimensions, this is in good alignment with the theory.

\begin{table}
	\centering
	\begin{tabular}{l|c|c|c}
		$k$&$8$&$16$&$32$\\\hline
		iterations& $16$&$7$&$4$\\
		residual&$2.487\cdot 10^{-7}$&$3.128\cdot 10^{-7}$&$4.591\cdot 10^{-7}$
	\end{tabular}
	\caption{Number of ``frozen nonlinearity'' iterations and final residuals for different values of $k$.}
	\label{tab:frozenit-k}
\end{table}

As next step, we investigate the dependence of \eqref{eq:froznonlinitcont} on the wave number $k$. We fix $\varepsilon$, $f$, and $g$ as above, set $h=2^{-6}$ and $p=3$.
Table~\ref{tab:frozenit-k} clearly shows a decrease of the required iterations (until the residual is below $5\cdot 10^{-7}$) with growing wave number. This indicates that the $k$-independent contraction factor for $d=2$ in Proposition \ref{prop:exfemnonlin} seems to be sub-optimal. Our results would suggest that the contraction factor may even decrease like $k^{-1}$. However, we also note that in all the experiments on the $h$, $p$ and $k$-dependence of the nonlinear iteration so far, the contraction factors $\sigma^{(l)}$ varied rather considerably over the iteration number $l$ in each experiment. For this reason, we gave these results in terms of the required iterations. 
The results of Tables~\ref{tab:frozenit-hp} and \ref{tab:frozenit-k} are qualitatively the same also for the iteration \eqref{eq:simplenewtonitcont} and therefore omitted here.

\begin{figure}
	\centering
	\includegraphics[width=0.9\textwidth, trim=25mm 5mm 25mm 10mm, clip=true]{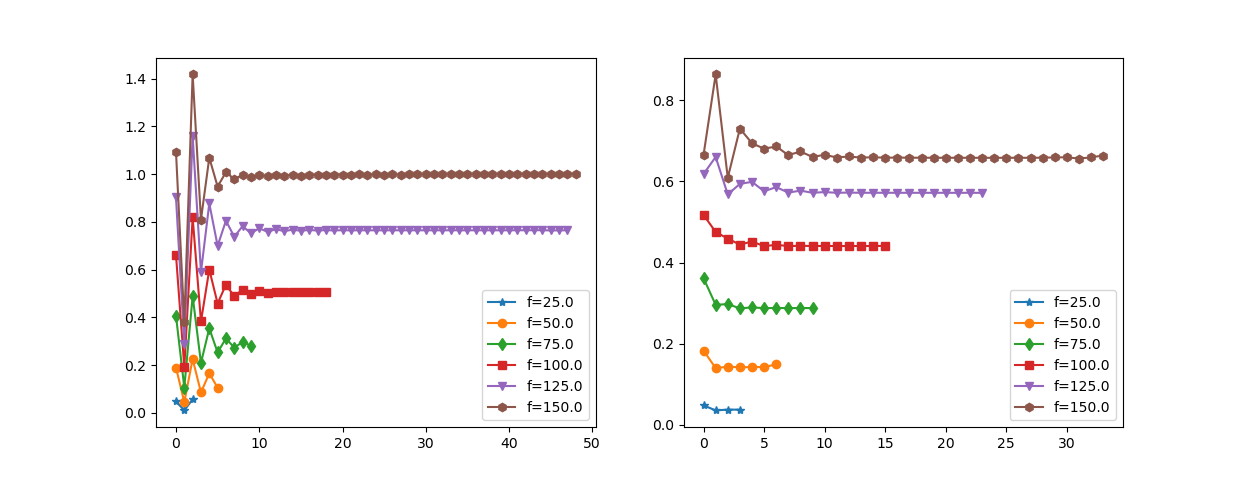}
	\caption{Contraction factors $\sigma^{(l)}$ over iteration number $l$ for different values of $f$ and the iteration schemes \eqref{eq:froznonlinitcont} (left) and \eqref{eq:simplenewtonitcont} (right).}
	\label{fig:iterationsf}
\end{figure}

We are mainly interested in how the convergence of the iteration depends on the $L^2$-norm of $f$ and on the size of $\varepsilon$. We fix $k=16$, $h=2^{-5}$ and $p=2$ and let the schemes \eqref{eq:froznonlinitcont} and \eqref{eq:simplenewtonitcont} iterate until either the (relative) residual is below $5\cdot 10^{-7}$ or the maximum number of $50$ iterations is reached.
Recall that we choose $f$ as a constant function on the whole domain for this experiment.
Figure~\ref{fig:iterationsf} shows the contraction factors $\sigma^{(l)}$ for the iteration schemes \eqref{eq:froznonlinitcont} and \eqref{eq:simplenewtonitcont} for different values of $f$. We first observe that there is an initial phase with varying $\sigma^{(l)}$ before an almost constant contraction factor is reached. This constant limit regime numerically verifies that both iteration schemes are of linear order like  a fixed-point scheme.
Additionally, we see that the initial phase seems to be longer for the frozen nonlinearity scheme \eqref{eq:froznonlinitcont}.
Comparing the behavior across different values of $f$, we clearly see that the contraction factor grows with larger $f$ in accordance with Proposition \ref{prop:exfemnonlin}.

\begin{figure}
	\centering
	\includegraphics[width=0.9\textwidth, trim=20mm 0mm 25mm 10mm, clip=true]{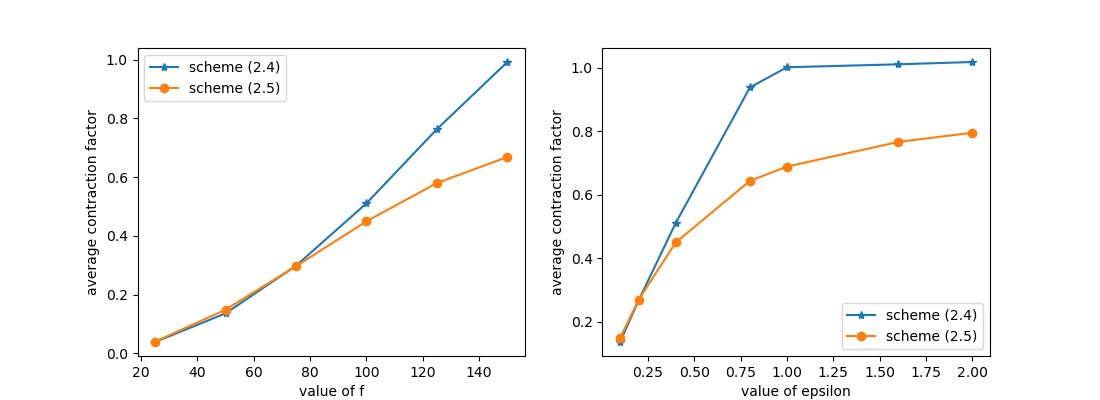}
	\caption{Average contraction factors over values of $f$ (left) and of $\varepsilon$ (right) for the two iteration schemes.}
	\label{fig:contraction}
\end{figure}

To make this better visible, we plot the average value of $\sigma$ across all iterations versus the value of $f$ in Figure \ref{fig:contraction} left. For small values of $f$, both iteration schemes perform similarly, but for larger $f$, \eqref{eq:simplenewtonitcont} has better contraction and convergence properties and, in that sense, is more robust. In particular, for $f\equiv 150$, the frozen nonlinearity \eqref{eq:froznonlinitcont} seems to be not converging (after $50$ iterations, the relative residual is still $0.15$), while the scheme \eqref{eq:simplenewtonitcont} converges.
In a similar spirit, we depict the average contraction factors over different values of $\varepsilon$ in Figure \ref{fig:contraction} right. Again, we see that the frozen nonlinearity is less robust than \eqref{eq:simplenewtonitcont}. The dependence on $\varepsilon$ seems in general to be more severe that the one on $f$. While the frozen nonlinearity converges only for $\varepsilon ={0.1, 0.2.0.4}$ in our example, the scheme \eqref{eq:simplenewtonitcont} converges for all considered values up to $\varepsilon=2$. For $\varepsilon=1.6$ and $\varepsilon=2$, $50$ iterations did not suffice to reach the residual of $5\cdot 10^{-7}$, but the contraction factors clearly suggest that \eqref{eq:simplenewtonitcont} should be able to reach that tolerance if we allowed more iterations. In fact, we obtain a final residual of $5.4\cdot 10^{-7}$ for $\varepsilon=1.6$ and of $2.3\cdot 10^{-6}$ for $\varepsilon=2$.

Summarizing, we could numerically confirm the better ``robustness'' of scheme \eqref{eq:simplenewtonitcont} over the frozen nonlinearity with respect to $\varepsilon$ and the (right-hand side) data, cf.~\cite{YuaL17}. Unfortunately, we could not reflect this in our theory, where the assumption on the data are more restrictive for the iteration \eqref{eq:simplenewtonitcont}. Further, our experiments indicate that the $k$-dependence of the contraction factor may be relaxed from the theoretical prediction in Proposition \ref{prop:exfemnonlin} and \cite{WuZ18}.

\section*{Conclusion}
In this contribution, we studied the finite element method with arbitrary but fixed polynomial degree for the nonlinear Helmholtz equation. By employing an error analysis for a linearized Helmholtz problem with small perturbation of the wave speed, we showed well-posedness and a priori error estimates under a smallness of the data assumption and the resolution condition $k(kh)^{p}\lesssim 1$.
In the treatment of the nonlinearity, we considered two different iteration schemes which can both be interpreted as fixed-point iterations.
Our numerical experiments confirm the theoretical estimates and, moreover, indicate that the results on the $hp$-FEM can be transferred from the linear case as well. Additionally, we compared the two iteration schemes concerning the performance and robustness with respect to the smallness of data assumption.
The contribution leaves some interesting future research questions, namely on a true $hp$-FEM analysis -- our constants may depend on the polynomial degree presently -- and on the different robustness of the two iteration schemes.

\section*{Acknowledgments}
Funding by Klaus-Tschira-Stiftung as well as by Deutsche Forschungsgemeinschaft (DFG, German Research Foundation) under Project-ID 258734477 (SFB 1173) and VE1397/2-1 is gratefully acknowledged.
The author would like to thank Céline Torres (University of Maryland) for fruitful discussions on the subject, especially concerning the solution splitting in Proposition \ref{prop:splitting}.

\appendix
\section{Proof of Proposition \ref{prop:splitting}}

\begin{proof}[Proof of Proposition \ref{prop:splitting}]
	Let $L_\Omega$, $L_\Gamma$, $H_\Omega$, $H_\Gamma$ be the high- and low-frequency filters on $\Omega$ and $\Gamma$, respectively, as introduced in \cite[Sec.~4.1.1]{MelS11}.
	For convenience, we briefly repeat the construction. Let $\mathcal{F}$ be the Fourier transform on $\mathbb{R}^d$ and set 
	\[L_{\mathbb{R}^d}(f)\coloneqq\mathcal{F}^{-1}(\chi_{\eta k}\mathcal{F}(f)), \quad H_{\mathbb{R}^d}(f)=(1-L_{\mathbb{R}^d})(f)\]
	for any $f\in L^2(\mathbb{R}^d)$, where $\chi_{\eta k}$ denotes the indicator function of $B_{\eta k}(0)$ and $\eta$ is a free parameter. For $\Omega$, let $E:L^2(\Omega)\to L^2(\mathbb{R}^d)$ be the Stein extension operator and define $H_\Omega (f)\coloneqq (H_{\mathbb{R}}^dE(f))|_\Omega$ and $L_\Omega$ analogously. For $\Gamma$, let $G:H^s(\Gamma)\to H^{3/2+s}(\Omega)$ denote the lifting operator with $\partial_\nu G(g)= g$ and define $H_\Gamma(g)\coloneqq\partial_\nu H_\Omega(G(g))$ and, analogously $L_\Gamma$.
	
	Denote by $N_k$, $S_k^\Delta$, and $S_k$ the solution operators for the linear, constant-coefficient Helmholtz equation (i.e., for $\varepsilon\equiv 0$) as introduced in \cite{MelS11}. More precisely, for $f\in L^2(\Omega)$, $N_k(f)$ is the unique solution to the Helmholtz equation $(-\Delta-k^2) u=f$ in $\mathbb{R}^d$ with Sommerfeld radiation condition.
	For $g\in L^2(\Gamma)$, $S_k^\Delta(g)$ is the solution to the ``good'' Helmholtz equation $(-\Delta+k^2)u=0$ in $\Omega$ with inhomogeneous Robin boundary conditions $\partial_\nu u+iku =g$ on $\Gamma$.
	Finally, $S_k(f,g)$ is the solution of the standard Helmholtz equation $(-\Delta-k^2) u=f$ in $\Omega$ with Robin boundary condition $\partial_\nu u+iku =g$ on $\Gamma$.
	Note that the sign for $ik$ in the boundary conditions is flipped in comparison to \cite{MelS11}, but the estimates remain valid.
	
	By the linearity of \eqref{eq:auxpb}, it suffices to prove Proposition \ref{prop:splitting} with only volume data $f$ or boundary data $g$ separately.
	We show how the assertion follows from the well-known results in case $g=0$, the other case can be proven similar.
	We set $w_\mathcal{A}^I:=S_k(L_\Omega f, 0)$, $w_{H^2}^I:=N_k(H_\Omega f)$.
	Denoting $\tilde{g}:=-\partial_n w_{H^2}^I-ikw_{H^2}^I$, we further set $w_\mathcal{A}^{II}:=S_k(0, L_\Gamma \tilde{g})$ and $w_{H^2}^{II}:=S_k^\Delta(H_\Gamma \tilde{g})$.
	By the definition of the solution operators, we deduce that the remainder $\tilde{r}:=w-w_\mathcal{A}^I-w_{H^2}^I-w_\mathcal{A}^{II}-w_{H^2}^{II}$ solves
	\begin{align*}
	-\Delta \tilde{r}-k^2(1+\chi_D \varepsilon|\Phi|^2)\tilde{r}=\tilde{f} \quad\text{in }\Omega, \qquad \partial_\nu \tilde{r}+ik \tilde{r}=0\quad \text{on }\Gamma,\\
	\text{with}\quad \tilde{f}:=k^2(2w_{H^2}^{II}+\chi_D\varepsilon|\Phi|^2(w_{H^2}^I+w_\mathcal{A}^I+w_{H^2}^{II}+w_\mathcal{A}^{II})).
	\end{align*}
	
	By the estimates for $w_{H^2}^{I, II}$ and $w_\mathcal{A}^{I, II}$ from \cite{MelS11}, we obtain
	\begin{align*}
	\|\tilde{f}\|_0&\leq 2k \|w_{H^2}^{II}\|_{1,k}+k\varepsilon\|\Phi\|^2_{L^\infty(D)}\bigl(\|w_\mathcal{A}^I\|_{1,k}+\|w_\mathcal{A}^{II}\|_{1,k}+\|w_{H^2}^I\|_{1,k}+\|w_{H^2}^{II}\|_{1,k}\bigr)\\
	&\leq 2q \|f\|_0+k\varepsilon\|\Phi\|^2_{L^\infty(D)} (2C\|f\|_0+2qk^{-1}\|f\|_0),
	\end{align*}
	where $q$ can be chosen arbitrarily small by adjusting $\eta$ above and $C$ is a $k$- and $\Phi$-independent constant.
	Here, we implicitly used that $\Omega$ is star-shaped and, hence, the stability constant of the Helmholtz equation with $\varepsilon=0$ is of the order one, cf. \cite{Mel95,CuFe2006}.
	Clearly, we see the existence of a constant $\theta_2$ such that, if $k\varepsilon\|\Phi\|_{L^\infty(D)}^2\leq \theta_2$, we have $\|\tilde{f}\|_0<\tilde{q}\|f\|_0$ for some $\tilde{q}<1$.
	Iterating this argument, we can write $w$ as sum of series (one series of analytic functions, one series of $H^2$-functions) that can be bounded with the help of the geometric series.
\end{proof}

\section{Proof of Lemma \ref{lem:erroraux}}
\begin{proof}[Proof of Lemma \ref{lem:erroraux}]
	We proceed similar to \cite{DuWu2015}. Let $P_h$ be defined via \eqref{eq:defellproj} and write $w-w_h=w-\overline{P_h\overline{w}}+w_h-\overline{P_h\overline{w}} = \rho+\eta_h$. We have by \eqref{eq:propellproj} and the solution splitting for $w$
	\[\|\rho\|_{1,k}\lesssim \inf_{v_h\in V_{h,p}}\|w-v_h\|_{1,k}\lesssim (h+(kh)^p) \Cdata,\]
	where we used the approximation properties of $V_{h,p}$ and Proposition \ref{prop:splitting} in the last step.
	Hence, we only have to consider $\eta_h=w_h-\overline{P_h\overline{w}}$ in the following.
	Observe that $\eta_h$ satisfies
	\[(\nabla \eta_h, \nabla v_h)-(k^2(1+\chi_D \varepsilon|\Phi|^2)\eta_h, v_h)+ik(\eta_h, v_h)_\Gamma = (k^2(1+\chi_D\varepsilon|\Phi|^2)(P_h w-w), v_h)\]
	for all $v_h \in V_{h,p}$.
	
	\emph{First step:} Insert $v_h=\eta_h$ and consider the imaginary part. We obtain with the standard $L^2$-projection $\Pi_h$ that
	\begin{align*}
	k\|\eta_h\|_\Gamma^2&= \Im\{(k^2(1+\chi_D\varepsilon|\Phi|^2)\rho, \eta_h)\}\leq k^2\|\Pi_h\rho\|_{1-p,h}\|\eta_h\|_{p-1,h}+k\theta_2 \|\rho\|_0\,\|\eta_h\|_0
	\end{align*}
	As in \cite[p.~792]{DuWu2015}, we have $\|\Pi_h\rho\|_{1-p, h}\lesssim h^p\|\rho\|_{1,k}$. Further, \eqref{eq:propellproj} gives $\|\rho\|_0\lesssim h\|\rho\|_{1,k}$. Hence,
	\[\|\eta_h\|_{0,\Gamma}^2\leq k^2h^{2p-1}\|\eta_h\|_{p-1,h}^2+\theta_2^2h\|\eta_h\|_0^2+h\|\rho\|_{1,k}^2.\]
	
	\emph{Second step:} Recall the definition of $A_h$ in \eqref{eq:Ah}.
	Consequently, we have
	\[(A_h \eta_h, v_h)=(k^2+1)(\eta_h, v_h)+k^2(\varepsilon|\Phi|^2 \eta_h, v_h)_D+ik(\eta_h, v_h)_\Gamma+k^2(\Pi_h\rho, v_h)+k^2(\varepsilon|\Phi|^2\rho, v_h)_D\]
	for any $v_h\in V_{h,p}$.
	For given $1\leq m\leq p$, set $v_h=A_h^{m-1}\eta_h$ to obtain
	\begin{align*}
	\|\eta_h\|_{m,h}^2&=(k^2+1)\|\eta_h\|_{m-1,h}^2+k^2(\varepsilon|\Phi|^2 \eta_h, A_h^{m-1}\eta_h)_D+ik(\eta_h, A_h^{m-1}\eta_h)_\Gamma\\
	&\quad +k^2(A_h^{(m-1)/2}\Pi_h\rho, A_h^{(m-1)/2}\eta_h)+k^2(\varepsilon|\Phi|^2\rho, A_h^{m-1}\eta_h)_D.
	\end{align*}
	From trace inequalities (cf.~\cite{DuWu2015}) and \eqref{eq:discretenorm-inverse} we obtain
	\begin{align*}
	|(\eta_h, A_h^{m-1}\eta_h)_\Gamma|&\lesssim \|\eta_h\|_{0,\Gamma}\,  h^{-m+1/2}\|\eta_h\|_{m-1,h}\\
	&\lesssim (kh^{p-m}\|\eta_h\|_{p-1,h}+\theta_2 h^{1-m}\|\eta_h\|_0+h^{1-m}\|\rho\|_{1,k})\|\eta_h\|_{m-1,h}\\
	&\lesssim (k\|\eta_h\|_{m-1,h}+\theta_2 h^{1-m}\|\eta_h\|_0+h^{1-m}\|\rho\|_{1,k})\|\eta_h\|_{m-1,h}.
	\end{align*}
	Further, we have with \eqref{eq:discretenorm-inverse} that
	\begin{align*}
	|k^2(\varepsilon|\Phi|^2 \eta_h, A_h^{m-1}\eta_h)_D|&\lesssim k\theta_2 \|\eta_h\|_0\|\eta_h\|_{2m-2,h}\lesssim k\theta_2 \|\eta_h\|_0\, h^{1-m}\|\eta_h\|_{m-1,h}
	\end{align*}
	and 
	\begin{align*}
	|k^2(\varepsilon|\Phi|^2\rho, A_h^{m-1}\eta_h)_D|&\lesssim k\theta_2 \|\rho\|_0\,\|\eta_h\|_{2m-2,h}\lesssim kh\theta_2 \|\rho\|_{1,k}\, h^{1-m}\|\eta_h\|_{m-1,h}.
	\end{align*}
	Altogether, we have for $1\leq m\leq p$ that
	\begin{align*}
	\|\eta_h\|_{m,h}^2&\lesssim k^2\|\eta_h\|_{m-1,h}^2+k^2\|\Pi_h\rho\|_{m-1,h}\|\eta_h\|_{m-1,h}\\*
	&\quad +(k\|\eta_h\|_{m-1,h}+\theta_2 h^{1-m}\|\eta_h\|_0+h^{1-m}\|\rho\|_{1,k}+\theta_2 h^{2-m}\|\rho\|_{1,k})k\|\eta_h\|_{m-1,h},
	\end{align*}
	and by Young's inequality we obtain 
	\begin{align*}
	\|\eta_h\|_{m,h}\lesssim k\|\eta_h\|_{m-1,h}+k\|\Pi_h\rho\|_{m-1,h}+\theta_2 h^{1-m}\|\eta_h\|_0+(\theta_2h^{2-m}+h^{1-m})\|\rho\|_{1,k}.
	\end{align*}
	As in \cite{DuWu2015}, it holds that $k\|\Pi_h\rho\|_{m-1,h}\lesssim h^{1-m}\|\rho\|_{1,k}$, so that we deduce
	\begin{equation}\label{eq:discretemnormetah}
	\|\eta_h\|_{m,h}\lesssim k\|\eta_h\|_{m-1,h}+\theta_2 h^{1-m}\|\eta_h\|_0+(\theta_2h^{2-m}+h^{1-m})\|\rho\|_{1,k}.
	\end{equation}
	Recursively, we obtain
	\begin{equation*}
	\begin{aligned}
	\|\eta_h\|_{m,h}&\lesssim k^m\|\eta_h\|_0+\sum_{j=0}^{m-1}\big\{ k^j\theta_2^{1+j}h^{1-m+j}\|\eta_h\|_0+k^j (\theta_2^{1+j}h^{2-m+j}+h^{1-m+j})\|\rho\|_{1,k}\big\},
	\end{aligned}
	\end{equation*}
	where we used $kh\lesssim 1$. Obviously, if $k\varepsilon\|\Phi\|_{L^\infty(D)}^2\leq \theta_2$ with $\theta_2$ sufficiently small, we deduce
	\begin{equation}\label{eq:discretenormetah}
	\begin{aligned}
	\|\eta_h\|_{m,h}&\lesssim (k^m+\theta_2 h^{1-m})\|\eta_h\|_0+h^{1-m}\|\rho\|_{1,k},
	\end{aligned}
	\end{equation}
	
	\emph{Third step:} 
	Let $z\in H^1(\Omega)$ be the dual solution of 
	\begin{align*}
	-\Delta z-k^2(1+\chi_D\varepsilon|\Phi|^2)z&=\eta_h,\\
	\partial_\nu z-ikz&=0.
	\end{align*}
	Multiplying with $w_h-w=\rho+\eta_h$, we obtain with the definition of $P_h$ and Galerkin orthogonality
	\begin{align*}
	(\rho +\eta_h, \eta_h)&=(\nabla (w_h-w), \nabla z)-(k^2(1+\chi_D\varepsilon|\Phi|^2)(w_h -w), z)+ik(w_h-w, z)_\Gamma\\
	&=(\nabla (w_h-w), \nabla(z- P_h z))-(k^2(1+\chi_D\varepsilon|\Phi|^2)(w_h-w), z-P_hz)\\
	&\qquad +ik(w_h-w, z-P_h z)_\Gamma\\
	&=(\nabla \rho, \nabla (z-P_hz))+ik(\rho, z-P_h z)_\Gamma -k^2((1+\chi_D\varepsilon|\Phi|^2)(\eta_h+\rho), z-P_h z).
	\end{align*}
	Hence, we deduce
	\begin{align*}
	\|\eta_h\|_0^2&=(\nabla \rho, \nabla (z-P_hz))+ik(\rho, z-P_h z)_\Gamma -k^2((1+\chi_D\varepsilon|\Phi|^2)(\eta_h+\rho), z-P_h z)-(\rho, \eta_h)\\
	&\leq \|\rho\|_{1,k}\|z-P_hz\|_{1,k}+k\theta_2 \|\rho\|_0\,\|z-P_hz\|_0+\|\rho\|_0\,\|\eta_h\|_0+k\theta_2 \|\eta_h\|_0\,\|z-P_hz\|_0\\
	&\qquad+k^2|(\eta_h, z-P_hz)|.
	\end{align*}
	Using the splitting according to Proposition \ref{prop:splitting} for $z$ and the properties of $P_h$, we have
	\begin{align*}
	\|P_h z-z \|_{1,k} &\lesssim (h+(kh)^p) \|\eta_h\|_0\\
	\|P_h z-z\|_0&\lesssim h(h+(kh)^p) \|\eta_h\|_0.
	\end{align*}
	Inserting these estimates as well as $\|\rho\|_0\lesssim h\|\rho\|_{1,k}$ into the one for $\eta_h$, we get
	\begin{align*}
	\|\eta_h\|_0^2&\lesssim \bigl((1+\theta_2)(h+(kh)^p)\|\rho\|_{1,k}\bigr) \|\eta_h\|_0+\theta_2 (h+(kh)^p)\|\eta_h\|_0^2+k^2|(\eta_h, z-P_h z)|.
	\end{align*}
	As in \cite[eq.~(5.13)]{DuWu2015}, we have 
	\begin{align*}
	|(\eta_h, z-P_h z)|&=|(\eta_h, \Pi_hz-P_hz)|\\*
	&\lesssim \|\eta_h\|_{p-1, h}\|\Pi_h z-z+z-P_h z\|_{1-p, h}\\
	&\lesssim \|\eta_h\|_{p-1, h}h^p(h+(kh)^p)\|\eta_h\|_0.
	\end{align*}
	Finally, we arrive at
	\begin{align*}
	\|\eta_h\|_0\lesssim (1+\theta_2)(h+(hk)^p)\|\rho\|_{1,k}+\theta_2 (h+(kh)^p)\|\eta_h\|_0+k^2h^p(h+(kh)^p)\|\eta_h\|_{p-1,h}.
	\end{align*}
	Obviously, if $k\varepsilon\|\Phi\|^2_{L^\infty(D)}\leq \theta_2$ with $\theta_2$ sufficiently small,
	\begin{align}\label{eq:l2normetah}
	\|\eta_h\|_0\lesssim (h+(hk)^p)\|\rho\|_{1,k}+k^2h^p(h+(kh)^p)\|\eta_h\|_{p-1,h}.
	\end{align}
	
	\emph{Fourth step:} By plugging \eqref{eq:discretenormetah} with $m=p-1$ into \eqref{eq:l2normetah}, we deduce
	\begin{align*}
	\|\eta_h\|_0&\lesssim (h+(hk)^p)\|\rho\|_{1,k}+k^2h^p(h+(kh)^p)\bigl((k^{p-1}+\theta_3 h^{2-p})\|\eta_h\|_0+h^{2-p}\|\rho\|_{1,k}\Bigr)\\
	&\lesssim (h+(kh)^p)\|\rho\|_{1,k}+((kh)^{p+1}+k(kh)^{2p})\|\eta_h\|_0+\theta_3 (h+(kh)^p)\|\eta_h\|_0,
	\end{align*}
	where we used $kh\lesssim 1$
	Hence, if $k(kh)^{2p}\leq C_0$ sufficiently small and $\theta_3\lesssim 1$ sufficiently small,
	we have
	\[\|\eta_h\|_0\lesssim (h+(kh)^p)\|\rho\|_{1,k}\lesssim (h^2+h(kh)^p+(kh)^{2p})\Cdata\]
	by the stability of $w$ and $P_h$.
	Combining with \eqref{eq:discretemnormetah} for $m=1$, we also obtain the bound for $\|\eta_h\|_{1,k}$, which finishes the proof of \eqref{eq:erraux}.
	
	The triangle inequality and the stability of $w$ then also give us \eqref{eq:stabfemaux} as well as existence and uniqueness of the discrete solution.
\end{proof}


\begin{thebibliography}{10}
	
	\bibitem{BarFT09}
	G.~Baruch, G.~Fibich, and S.~Tsynkov.
	\newblock A high-order numerical method for the nonlinear {H}elmholtz equation
	in multidimensional layered media.
	\newblock {\em J. Comput. Phys.}, 228(10):3789--3815, 2009.
	
	\bibitem{BeChMe2021}
	M.~Bernkopf, T.~Chaumont-Frelet, and J.~M. Melenk.
	\newblock Wavenumber-explicit stability and convergence analysis of hp finite
	element discretizations of {H}elmholtz problems in piecewise smooth media.
	\newblock arXiv preprint 2209.03601, 2022.
	
	\bibitem{ChNi2020}
	T.~Chaumont-Frelet and S.~Nicaise.
	\newblock Wavenumber explicit convergence analysis for finite element
	discretizations of general wave propagation problems.
	\newblock {\em IMA J. Numer. Anal.}, 40(2):1503--1543, 2020.
	
	\bibitem{CuFe2006}
	Peter Cummings and Xiaobing Feng.
	\newblock Sharp regularity coefficient estimates for complex-valued acoustic
	and elastic {H}elmholtz equations.
	\newblock {\em Math. Models Methods Appl. Sci.}, 16(1):139--160, 2006.
	
	\bibitem{DuWu2015}
	Yu~Du and Haijun Wu.
	\newblock Preasymptotic error analysis of higher order {FEM} and {CIP}-{FEM}
	for {H}elmholtz equation with high wave number.
	\newblock {\em SIAM J. Numer. Anal.}, 53(2):782--804, 2015.
	
	\bibitem{EveW14}
	G.~Ev\'{e}quoz and T.~Weth.
	\newblock Real solutions to the nonlinear {H}elmholtz equation with local
	nonlinearity.
	\newblock {\em Arch. Ration. Mech. Anal.}, 211(2):359--388, 2014.
	
	\bibitem{GaSp23}
	Jeffrey Galkowski and Euan~A. Spence.
	\newblock Sharp preasymptotic error bound for the {H}elmholtz $h$-{FEM}.
	\newblock arXiv preprint 2301.03574, 2023.
	
	\bibitem{GiTr2001}
	David Gilbarg and Neil~S. Trudinger.
	\newblock {\em Elliptic partial differential equations of second order}.
	\newblock Classics in Mathematics. Springer-Verlag, Berlin, 2001.
	\newblock Reprint of the 1998 edition.
	
	\bibitem{GolG84}
	J.~A. Goldstone and E.~Garmire.
	\newblock Intrinsic optical bistability in nonlinear media.
	\newblock {\em Phys. Rev. Lett.}, 53:910--913, 1984.
	
	\bibitem{GraS20}
	I.~G. Graham and S.~A. Sauter.
	\newblock Stability and finite element error analysis for the {H}elmholtz
	equation with variable coefficients.
	\newblock {\em Math. Comp.}, 89(321):105--138, 2020.
	
	\bibitem{JiLiWuZo2022}
	Run Jiang, Yonglin Li, Haijun Wu, and Jun Zou.
	\newblock Finite element method for a nonlinear {PML} {H}elmholtz equation with
	high wave number, 2022.
	
	\bibitem{Ker75}
	J.~Kerr.
	\newblock A new relation between electricity and light: {D}ielectrified media
	birefringent.
	\newblock {\em Philosophical Magazine}, 50:337--348, 1875.
	
	\bibitem{LaSpWu2022}
	D.~Lafontaine, E.~A. Spence, and J.~Wunsch.
	\newblock A sharp relative-error bound for the {H}elmholtz {$h$}-{FEM} at high
	frequency.
	\newblock {\em Numer. Math.}, 150(1):137--178, 2022.
	
	\bibitem{LaSpWu2021a}
	D.~Lafontaine, E.~A. Spence, and J.~Wunsch.
	\newblock Wavenumber-explicit convergence of the {$hp$}-{FEM} for the
	full-space heterogeneous {H}elmholtz equation with smooth coefficients.
	\newblock {\em Comput. Math. Appl.}, 113:59--69, 2022.
	
	\bibitem{LaSpWu2021b}
	David Lafontaine, Euan~A. Spence, and Jared Wunsch.
	\newblock Decompositions of high-frequency {H}elmholtz solutions via functional
	calculus, and application to the finite element method.
	\newblock arXiv preprint 2102.13081, 2021.
	
	\bibitem{MaiV22}
	Roland Maier and Barbara Verf\"urth.
	\newblock Multiscale scattering in nonlinear {K}err-type media.
	\newblock Math.~Comp.~(online first), 2022.
	
	\bibitem{Mel95}
	J.~M. Melenk.
	\newblock {\em On generalized finite-element methods}.
	\newblock ProQuest LLC, Ann Arbor, MI, 1995.
	\newblock PhD Thesis, University of Maryland, College Park.
	
	\bibitem{MelS10}
	J.~M. Melenk and S.~Sauter.
	\newblock Convergence analysis for finite element discretizations of the
	{H}elmholtz equation with {D}irichlet-to-{N}eumann boundary conditions.
	\newblock {\em Math. Comp.}, 79(272):1871--1914, 2010.
	
	\bibitem{MelS11}
	J.~M. Melenk and S.~Sauter.
	\newblock Wavenumber explicit convergence analysis for {G}alerkin
	discretizations of the {H}elmholtz equation.
	\newblock {\em SIAM J. Numer. Anal.}, 49(3):1210--1243, 2011.
	
	\bibitem{Pem2020}
	Owen~Rhys Pembery.
	\newblock {\em The Helmholtz Equation in Heterogeneous and Random Media:
		Analysis and Numerics}.
	\newblock PhD thesis, University of Bath, 2020.
	
	\bibitem{SchW1977}
	A.~H. Schatz and L.~B. Wahlbin.
	\newblock Interior maximum norm estimates for finite element methods.
	\newblock {\em Math. Comp.}, 31(138):414--442, 1977.
	
	\bibitem{Sch1997}
	Joachim Sch{\"o}berl.
	\newblock Netgen an advancing front 2d/3d-mesh generator based on abstract
	rules.
	\newblock {\em Computing and visualization in science}, 1(1):41--52, 1997.
	
	\bibitem{Sch2014}
	Joachim Sch{\"o}berl.
	\newblock C++ 11 implementation of finite elements in ngsolve.
	\newblock {\em Institute for analysis and scientific computing, Vienna
		University of Technology}, 30/2014, 2014.
	
	\bibitem{WuZ18}
	H.~Wu and J.~Zou.
	\newblock Finite element method and its analysis for a nonlinear {H}elmholtz
	equation with high wave numbers.
	\newblock {\em SIAM J. Numer. Anal.}, 56(3):1338--1359, 2018.
	
	\bibitem{XuB10}
	Z.~Xu and G.~Bao.
	\newblock A numerical scheme for nonlinear {H}elmholtz equations with strong
	nonlinear optical effects.
	\newblock {\em J. Opt. Soc. Am. A}, 27(11):2347--2353, 2010.
	
	\bibitem{YuaL17}
	L.~Yuan and Y.~Y. Lu.
	\newblock Robust iterative method for nonlinear {H}elmholtz equation.
	\newblock {\em J. Comput. Phys.}, 343:1--9, 2017.
	
	\bibitem{ZhWu2013}
	Lingxue Zhu and Haijun Wu.
	\newblock Preasymptotic error analysis of {CIP}-{FEM} and {FEM} for {H}elmholtz
	equation with high wave number. {P}art {II}: {$hp$} version.
	\newblock {\em SIAM J. Numer. Anal.}, 51(3):1828--1852, 2013.
	
\end{thebibliography}
\end{document}